\documentclass{article}

\usepackage{amsmath}
\usepackage{amssymb}
\usepackage{amsthm}
\usepackage{enumerate}
\usepackage{breqn}
\usepackage{cleveref}
\usepackage[usenames,dvipsnames]{xcolor}
\usepackage{mathrsfs}
\usepackage{ulem}

\newtheorem{theorem}{Theorem} 
\newtheorem*{theorem*}{Theorem \ref{thm: main}}
\newtheorem{lemma}{Lemma}[section]

\newtheorem{prop}{Proposition}[section]
\newtheorem{remark}{Remark}[section]

\everymath{\displaystyle}
\newcommand{\F}{\mathcal{F}}

\newcommand{\real}{\mathbb{R}}

\newcommand{\eqn}{\begin{equation}}
\newcommand{\eeqn}{\end{equation}}

\def\R{\real}

\def\g{\gamma}  
\def\ep{\epsilon}

\def\al{\alpha}

\def\pa{\partial}
\def\nb{\nabla}
 
\def\om{\omega}

\def\F{\mathcal F}

\def\R{\mathcal R}
\def\S{\mathscr S}
\def\W{\mathcal W}
\def\ka{\kappa}

\def\bcr{\begin{color}{red}}
\def\bcb{\begin{color}{blue}}
\def\bcg{\begin{color}{Green}}
\def\bcv{\begin{color}{violet}}
\def\ec{\end{color}}

\begin{document}

\title{Existence of Rotating Stars with Variable Entropy}
\date{}
%\maketitle

% classification 35Q35, 35Q85, 35R35

%\author{Juhi Jang} {(University of Southern California)}, 
%{Walter A. Strauss} {(Brown University)}
%and {Yilun Wu} {(University of Oklahoma)}

\author{Juhi Jang\thanks{Department of Mathematics, University of Southern California, Los Angeles, CA 90089, USA, and Korea Institute for Advanced Study, Seoul, Korea.  Email: juhijang@usc.edu.}, \ Walter A. Strauss\thanks{Department of Mathematics and Lefschetz Center for Dynamical Systems, Brown University, Providence, RI 02912, USA, Email: walter\_strauss@brown.edu.}, \ and Yilun Wu\thanks{Department of Mathematics, University of Oklahoma, Norman, OK 73069, USA, Email: allenwu@ou.edu.}}

\maketitle

\begin{abstract}  
We model a rotating star as a compressible fluid subject to gravitational forces.  
In almost all the mathematical literature the entropy is considered to be constant.  Here we allow it to be variable.  We consider a star that steadily rotates differentially around a fixed axis, say the $z$-axis.  We prove the existence of a family of such stars with small angular velocity $\om$ and small entropy variation $s$ and with an equation of state $p=Ke^s\rho^\g$.  Our analysis reduces to a hyperbolic equation for the modified entropy coupled to an elliptic equation for the modified density, together with a mass constraint.  
Due to the variable entropy and the consequent loss of both regularity and variational structure, all the methods in the previous literature fail.  We develop a new ad hoc perturbative strategy that allows us to construct rotating stars that bifurcate from the non-rotating ones. %We employ a new strategy that is not a standard perturbation argument.  

\end{abstract}

\section{Introduction} 
\subsection {Rotating Stars} 
The study of rotating stars is a classical topic in astrophysics and mathematics, and there has been a great deal of interest and activity for centuries. Early studies can be tracked back to Newton, Maclaurin, Jacobi, Poincar\'{e}, Liapunov {\it et al} who studied incompressible stars, while compressible stars began to be treated later by Lichtenstein \cite{Lichtenstein} and Chandrasekhar \cite{Chandra1933}. We refer to \cite{Chandrasekhar1967, jardetzky2013theories} for historical accounts on the topic. 
Other excellent general references are \cite{Tassoul,Hansen}.
The existence of rotating stars with a given angular velocity distribution is not a trivial task. In particular the support of a star is not known {\it a priori} and it is one of the unknowns. 

In the search for rotating star solutions there are two modern approaches. One is based on variational methods where the rotating solutions are obtained via a constrained minimization problem \cite{AuchmutyB, CaF, FrTur3, FLS, Li, LuoS, Wu16}, and the other is based on perturbative methods where the solutions are obtained by an implicit function theorem or contraction mapping principle around non-rotating Lane-Emden stars \cite{Heilig, JJTM, JJTM2, JSW19, SW17}.  
The second approach has been extended to construct a global set of rotating stars 
\cite{SW19}.  
Both approaches have been particularly successful to prove the existence of rotating star solutions with constant entropy, in which case an integral reformulation can be effectively used.

When it comes to rotating star solutions with variable entropy,  however, 
very few results are available, despite its physical importance \cite{Ireland, Tassoul}.   
%{\color{blue}{The role of entropy is physically important; see for instance \cite{Ireland}.}} 
The goal of this paper is to establish the existence of rotating star solutions  for self-gravitating perfect fluids with variable entropy.   
To the best of our knowledge, ours is the first mathematical result on rotating star solutions that allows variable entropy. Furthermore, ours is the first that allows the prescribed angular velocity $\om$ to depend on both of the cylindrical coordinates $r$ and $z$, while all prior works treated merely $r$-dependent angular velocity profiles. 
A system with nontrivial $z$ dependence is called a {\it barocline}.  
 Actual stellar rotations typically exhibit such dependence.  An example is our own sun  \cite{Tassoul}. 
 
 Our problem does not have a variational structure.  Although the implicit function theorem would seem like a natural technique, it does not work because of  the loss of regularity.  Instead we use an explicit iteration and we perform estimates in two different normed spaces  in Sections 4 and 5.  In order to make optimal use of the available regularity, most of our analysis is carried out in a bounded domain.    

\subsection{Euler-Poisson Model} 
Now we describe the main contents of the current work.    
We regard the star as a compressible inviscid fluid.  
The basic equations are those of Euler-Poisson which describe the conservation of mass, momentum and energy 
subject to self-gravity and the transport of entropy by the fluid.  
The equations are: 
\begin{align}  
&\rho_t + \nb\cdot(\rho \bold v) = 0, \label{eq: mass conserv}\\
&(\rho \bold v)_t + \nb\cdot(\rho \bold v \otimes \bold v) + \nb p = \rho \nb U  ,\label{eq: momentum conserv} \\
&s_t + \bold v\cdot \nb s = 0  , \label{eq: s transport}\\
&U(x,t) = G\int_{\real^3} \frac{\rho(x',t)}{|x-x'|} dx'  , \label{eq: g law}
\end{align} 
where $\rho$ is the density, $\mathbf v$ the velocity, $p$ the pressure, $s$ the entropy,  and $G$  Newton's universal constant of gravitation.  
The first three equations are valid in the region $\{\rho>0\}$ occupied by the star.
We assume the pressure $p$ is a function of  $\rho$ and $s$, specifically that  
the star is a  simple ideal gas:
\eqn \label{eq: state}
p = Ke^s \rho^\gamma,   \eeqn
as in \cite{CoFr} (p.6-7), \cite{Callen} (Section 3.4) and \cite{Fermi}(p. 61-62),  
where $K>0$ and $\gamma>1$ are constants.  
Note that the entropy transport equation \eqref{eq: s transport} is a consequence of the other conservation laws together with the energy conservation law:
\eqn
\left(\rho\left(\tfrac12|\mathbf v|^2+\varepsilon \right)\right)_t+\nabla\cdot \left(\rho\left(\tfrac12 |\mathbf v|^2+\varepsilon \right)\mathbf v\right)+\nabla \cdot (p\mathbf  v)=\rho\nabla U\cdot\mathbf v,
\eeqn
where 
$$\varepsilon=\tfrac{K}{\gamma-1}e^s\rho^{\gamma-1}$$
is the specific internal energy.
For notational convenience we choose $K = \tfrac{\gamma}{\gamma-1}$ and $G=1$ in the following, but the same results hold for any positive $K$ and $G$. 

Our main result states that there exists a family of rotating solutions 
which are axisymmetric and even in $z$ and which bifurcate from the 
non-rotating (Lane-Emden) radial solution $\rho^0$ that has a radius $R^0$ and zero entropy. 
(See Lemma \ref{lem: LE} for details of such non-rotating star solutions.) We denote by $B_{R^0}$ the ball with radius $R^0$ centered at the origin on which $\rho^0$ is supported. In the following theorem, let $\beta\in \left(0,\min\left(\tfrac{2-\gamma}{\gamma-1},1\right)\right)$ be given. 

\begin{theorem}\label{thm: intro main}
Assume $\tfrac65<\gamma<2$, and $\gamma\ne \tfrac43$. Let $B$ be an arbitrary open ball strictly containing $B_{R^0}$. Let $s_0\in C^{1,\beta}(\overline{B}\cap\{z=0\})$ be a given axisymmetric floor entropy profile on the equatorial plane $\{z=0\}$, and $\om\in C^{2,\beta}(\overline{B})$ be a given axisymmetric angular velocity profile that is even in $z$. Then there exists a family of solutions $(\rho_{\kappa,\mu},s_{\kappa,\mu})\in [C^{1,\beta}(\overline{B})]^2$ for $\kappa,\mu>0$ sufficiently small, which depend continuously on $(\kappa,\mu)$ and converge to $(\rho^0,0)$ as $(\kappa,\mu)\to 0$ such that each solution satisfies the following properties.
\begin{itemize}
\item $(\rho_{\kappa,\mu},s_{\kappa,\mu})$ is axisymmetric and even in $z$.
\item $(\rho_{\kappa,\mu},s_{\kappa,\mu})$ is a time independent solution of \eqref{eq: mass conserv}-\eqref{eq: state}, where the velocity $\mathbf v$
is an axisymmetric rotation with angular velocity $\sqrt{\kappa}~\omega(r,z)$.
\item $s_{\kappa,\mu}$ satisfies the floor entropy condition $$s_{\kappa,\mu}\big|_{z=0}=\mu s_0.$$
\item $\rho_{\kappa,\mu}$ is nonnegative, compactly supported in $B$ and has the same total mass
$$\int_{B}\rho_{\kappa,\mu}~dx=\int_{B_{R^0}}\rho^0~dx$$
as the Lane-Emden solution.
\end{itemize}
\end{theorem}

%One can see from the statement of 
In Theorem \ref{thm: intro main}, the parameters $k$ and $\mu$ characterize the respective intensity of rotation and floor entropy. The exceptional index $\gamma=4/3$ is excluded from the theorem. It is the mass critical index for which all the radial Lane-Emden solutions have the identical total mass. In this exceptional case there is a trivial curve of non-rotating solutions which are related to each other by scaling. Moreover, in case 
$\gamma=4/3$ with constant entropy and angular velocity, it is shown in \cite{SW20}  that there does not exist any slowly rotating solution close to a given Lane-Emden solution with the same total mass. So in that case the local solution space is exhausted by the trivial curve of non-rotating stars.

\subsection{Related Literature}  
So far as we know, the only works in the literature that are related to the Euler-Poisson equations with variable entropy are  \cite{DLY, DY, LuoS0, Wu15, Yuan21}.  
{\it Our result differs from them in several important ways.} First of all, most of the earlier works that considered variable entropy were essentially restricted to solving only the divergence of the momentum equation \eqref{eq: momentum conserv} on a given domain with  prescribed entropy $s$, while the curl part was largely ignored. Therefore such solutions do not necessarily satisfy the Euler-Poisson system. In fact, even if the curl part holds, as for instance in the case of radial symmetry, in order to recover the momentum equation from the divergence and the curl components, one would need an extra boundary condition (different from the Dirichlet condition of $\rho$) on the surface of the star. Without such a boundary condition, the two sides of the momentum equation would differ by the gradient of a harmonic function. See Section \ref{subsec: div curl} for a more detailed discussion. 
{\it This key step of going from the divergence and the curl back to the original system seems to be missing in all the prior works, with the exception of \cite{Yuan21}. }

In \cite{Yuan21}, the author uses a variational method on an integral formulation of the momentum equation, in a spirit analogous to the prior variational results \cite{AuchmutyB} with constant entropy. However, in order to fit variable entropy into this framework, the author has to work with functions that are constants on a given nested family of ellipsoids. As all the test functions for the Euler-Lagrange variation then have fixed ellipsoidal symmetry, the solutions obtained in \cite{Yuan21} merely solve ellipsoidal averages of the momentum equation and do not solve the original system. 

In summary, to the best of our knowledge, Theorem \ref{thm: intro main} is the first existence result of variable entropy rotating star solutions to the full Euler-Poisson system.  In order to obtain our result, we predetermine neither the entropy $s$ nor the fluid domain of the star. Both of them appear as unknowns in our formulation.

\subsection{Significance of the dependence on $z$}  
Moreover, all of the above mentioned works require an angular velocity profile $\omega$ that depends only on $r$. 
{\it In contrast, our result allows an arbitrary axisymmetric rotation speed $\omega$ that is even in $z$.} In our opinion, when allowing variable entropy, an $\omega$ with nontrivial $z$-dependence presents the more interesting case, and in a certain sense the only interesting case. In fact, as we shall see in Section \ref{subsec: div curl}, the curl of the momentum equation implies that $\omega^2$ depends only on $r$ if and only if $\nabla s$ and $\nabla \rho$ are  parallel vectors everywhere. In other words, if $\omega^2$ depends only on $r$, then $s$ must be constant on the level surfaces of $\rho$. This fact is known as the Poincar\'e-Wavre theorem in the physics literature (\cite{Tassoul}). If the solution is close to a nonrotating Lane-Emden star, all level surfaces of $\rho$ have different values. 
 In this case  %, the Poincar\'e-Wavre theorem implies that 
there exists a one-variable function $f$ such that $s=f(\rho)$ on the fluid domain. Such a solution is called a {\it barotrope}. 
%Therefore, for a barotrope, the entropy in the solution is not completely independent from the fluid density. One can basically dispense with $s$ by adopting a new equation of state 
Therefore the pressure $p$ also depends only on $\rho$. 
Finally, we point out that the rotation speed of our own sun %involves an $\omega$ that 
depends on both $r$ and $z$ (\cite{Tassoul}), so that our result provides the first mathematical description in the Euler-Poisson model of our own star.  

Before beginning to discuss our method, we set the convention here that all functions dealt with in this paper are axisymmetric and are even in the $z$ variable. Therefore, every function space we will use is assumed to have the above mentioned symmetry without further declaration. In particular, a $C^1$ function $V$ with such symmetry will have the properties that $\nabla V(0,0,0)=0$, $\pa_{x_1} V(0,0,z)=0$, $\pa_{x_2} V(0,0,z)=0$, $\partial_z V(x_1,x_2,0)=0$.  
% $\pa_x V(0,0,z)=0$, $\pa_y V(0,0,z)=0$, $\partial_z V(x,y,0)=0$. 
For the same reason, we identify a function defined for $r>0,z>0$ with a function defined on $\real^3$ with the above mentioned symmetry without further explanation.

%\sout{As we will point out later, if the entropy is constant then $\om$ cannot depend in $z$. There are a few studies of models where an entropy is prescribed in advance, including \cite{LuoS}, in which, however, the entropy does not satisfy \eqref{eq: s transport}, the support of the star is also prescribed and $\om$ is independent of $z$.  Deng et al paper? The paper Yuan-Yuan [ ] unfortunately has a gap. [ALLEN: how shall we say this?] MENTION LUO-SMOLLER MENTION ALLEN'S OLD PAPER}

%%%%%%%%%%%%%%%%%
\section{Reformulation}  

The main difficulties come from the presence of the nontrivial entropy.  First of all, the entropy cannot be prescribed arbitrarily but must be solved together with the other unknowns. In fact, by taking the curl of the momentum equation, one finds that the entropy satisfies a first order transport equation whose coefficients are given by the first derivatives of the density (or enthalpy).  The $z$ derivative of the angular velocity comes into the transport equation as a source term. This new transport structure highlights fundamental differences between constant and variable entropy.  For variable entropy, we are led to studying the ensuing system of equations which is of mixed elliptic-hyperbolic type, while for constant entropy we would be led to a purely elliptic problem. 
The nontrivial coupling between the entropy and the enthalpy precludes an integral reformulation.  

\subsection{Reduction to the Div-Curl System}\label{subsec: div curl}
As a first step to solving this problem, we must choose certain good variables that yield a system with better structure.  This is a crucial step in  our analysis.  
We use cylindrical coordinates $(r,z)$ where all the functions are independent of 
the cylindrical angle $\theta$.   
Steady flow means that all the time derivatives in \eqref{eq: mass conserv}-\eqref{eq: s transport} vanish.  
Thus the mass conservation equation \eqref{eq: mass conserv} is automatically satisfied.  
Rotation with angular velocity $\omega$ means that the velocity is 
$\bold v = \sqrt\ka\omega(r,z)r\bold e_\theta$, where $r\bold e_\theta=(-x_2,x_1,0)$.   
Then the entropy transport equation \eqref{eq: s transport} is automatically satisfied.  
The momentum equation \eqref{eq: momentum conserv} reduces to 
\eqn \label{eq: EP vector}
\frac{\nb p}\rho = \nb\left(\frac1{|\cdot|}*\rho\right) + \ka \omega^2r\bold e_r  , \eeqn 
where $r\bold e_r= (x_1,x_2,0)$. 
We introduce the new variables 
\eqn \label{eq: change of var}
S=e^{ s/\g}, \quad V= \rho^{\g-1} e^{(\g-1) s/\g}, \quad q = \frac1{\g-1}.    \eeqn 
Then \eqref{eq: EP vector} takes the form 
\eqn\label{eq: EP vector with V fluid domain}
S\nb V = \nb\left(\frac1{|\cdot|}*\frac{V^q} {S}\right)  +  \kappa\om^2r\bold e_r  . 
\eeqn
This equation is originally only required to hold on the fluid domain, namely  where $\rho>0$. 

However, we will look for a solution to the slightly modified equation
\eqn \label{eq: EP vector with V}
S\nb V = \nb\left(\frac1{|\cdot|}*\frac{V_+^q} {S}\right)  +  \kappa\om^2r\bold e_r,\eeqn
where $V_+=\max(V,0)$, in a large ball $B_R$ containing the fluid domain, 
where $R$ will be chosen later. 
The density $\rho$ will then be recovered as $V_+^q S^{-1}$. 
{\it In this way, the fluid domain is not prescribed in advance, but is obtained as the positivity set of the solution $V$.} 
Taking  the curl and the divergence of \eqref{eq: EP vector with V}, we obtain the two scalar equations 
\begin{align} 
S_rV_z-S_zV_r &= - \ka(\om^2 r)_z ,  \label{eq: curl}  \\
 \nb\cdot (S\nb V) &= -4\pi \frac{V_+^q}S + \ka\nb\cdot(\om^2r\mathbf e_r).\label{eq: div} 
 \end{align}
We will regard \eqref{eq: curl} as a transport equation for $S$ with coefficients given by components of $\nabla V$ and source term given by its right hand side.  Although we will treat \eqref{eq: curl} structurally as a transport equation, it has nothing directly to do with the actual entropy transport by fluid velocity given in \eqref{eq: s transport}.  
 Equation \eqref{eq: curl}  illustrates the effect of the dependence on $z$ of the angular velocity.  In fact, it is immediately clear that any $z$-dependence of $\om$ implies that the entropy is not constant.  

Equation \eqref{eq: div}, on the other hand, will be treated basically as a semilinear elliptic equation for $V$ with $S$ appearing as a coefficient. However, as we will see more clearly later, 
the coupling of $S$ with $V$ causes the whole problem to have a quasilinear flavor. 

We have to supplement \eqref{eq: curl} and \eqref{eq: div} with suitable boundary conditions. The first of these is the {\it floor entropy condition} (on the equatorial plane)
\eqn \label{eq: floor s}
S(r,0)=\exp\left({\frac{\mu s_0(r)}{\g}}\right),   \eeqn
where $s_0(r)$ is the floor entropy profile which we treat as essentially arbitrary, and $\mu$ is a small parameter. We will use this condition to uniquely solve the transport equation \eqref{eq: curl}.
Secondly, the bounded domain $B_R$ is employed in order to make maximal use of the elliptic regularity estimates.  Consequently,  we will need a key boundary condition of $V$ on $\partial B_R$ that will help us recover \eqref{eq: EP vector with V} from its curl and divergence. Such a boundary condition can be much simplified if we cut off $\omega^2$ and $s_0$ near $\partial B_R$. Such a cut-off will have no effect to our final solution. 

\subsection {Non-rotating Solution}
For clarification, let us be a bit more specific about the non-rotating Lane-Emden solution around which we perturb.
\begin{lemma}\label{lem: LE}
Given $\tfrac65<\gamma<2$ and $R^0>0$, there exists a unique function $V^0$ which is continuous on $\real^3$, positive in the ball $B_{R^0}$ and negative outside $B_{R^0}$, as well as a unique constant $\alpha^0$ such that
\eqn\label{eq: LE1}
V^0=\frac{1}{|\cdot|}*(V^0)_+^q+\alpha^0  
\eeqn
in all of $\real^3$.
Furthermore $(V^0,\alpha^0)$ has the following properties: $\alpha^0<0$, $V^0\in C^{3,\beta_0}(\real^3)$,  where $\beta_0=\min\left(q-1,1\right)$, and $V^0\big|_{B_{R^0}}\in C^\infty(B_{R^0})$. Furthermore, $V^0$ is radially symmetric and strictly decreasing, and 
\eqn\label{eq: V0 near 0}
\lim_{\tau\to 0^+}\frac{(V^0)'(\tau)}{\tau}<0.
\eeqn
Here $(V^0)'$ means the radial derivative of $V^0$.
\end{lemma}
\begin{proof}
The classical reference is Chapter IV of \cite{Chandra}. See also Lemmas 3.2 and 3.3 of \cite{SW17}.
\end{proof}

Obviously $V=V^0$ is a solution to \eqref{eq: EP vector with V} with $S\equiv 1$ and $\kappa=0$. Fixing a radius $R>R^0$, we will perturb $V^0$ to a solution $V$ corresponding to a rotating star. Since $V^0$ is strictly negative near $\partial B_R$, the support of $V_+$ will stay away from $\pa B_R$ if the perturbation is sufficiently small.  
The following useful observation is obvious.  
\begin{remark}\label{lem: support}
There exists a $\delta>0$ such that if $\|V-V^0\|_{C^0(\overline{B_R})}<\delta$, then the support of $V_+$ is contained in $B_{\frac{R^0+R}{2}}$.
\end{remark}
In our perturbative method, $V$ will always stay close to $V^0$ on $B_R$. Therefore, without loss of generality, we can first multiply $\omega^2$ and $s_0$ by a smooth cutoff function $\chi$ such that $\chi\equiv 1$ on $B_{\frac{R^0+R}{2}}$ and $\chi\equiv 0$ outside $B_{\frac{R^0+2R}{3}}$, and then solve the whole problem. Since the fluid domain will always be contained in $B_{\frac{R^0+R}{2}}$, the cutoff will be invisible to the solution. In the following, we will always assume that $\omega^2$ and $s_0$ are prepared in this way, and when appropriate, we will even extend these functions outside $B_R$ by zero. Such an extension is obviously smooth as well. 

\subsection{Boundary Condition and Main Theorem} 
As we shall see in Section \ref{sec: S est}, the solution $S$ to the transport equation \eqref{eq: curl} will be identically equal to 1 near $\partial B_R$, as a consequence of the cutoff we performed on $\omega^2$ and $s_0$. The same cutoff and \eqref{eq: EP vector with V} imply that 
\eqn
\nabla V=\nabla\left(\frac1{|\cdot|}*\frac{V_+^q}{S}\right)
\eeqn
near $\partial B_R$. We thus obtain the key {\it boundary condition} on $V$:
\eqn\label{eq: BC}
V=\frac{1}{|\cdot|}*\frac{V_+^{q}}{S}+\alpha \quad \text{ on }\partial B_R    
\eeqn
for some constant $\alpha$. Note that $V_+$ will be supported on $B_{\frac{R^0+R}{2}}$ by Remark \ref{lem: support}. Thus the function $V_+^qS^{-1}$ that appears under the convolution in \eqref{eq: BC} can be regarded as the zero extension from its values on $B_R$ to $\real^3$.

Finally, we prescribe the {\it mass constraint}  
\eqn \label{eq: mass}  
\int_{B_R} \frac{V_+^q}{S}\ dx = M=\int_{B_R}(V^0)_+^q~dx,
\eeqn  
so that our rotating star will have the same mass $M$ as the unperturbed Lane-Emden star.

Our method to work with the div-curl system is summarized in the following lemma.
%\bcr change this following lemma connect with Theorem 1 better\ec

\begin{lemma}\label{lem: div curl}
%Let $\omega^2\in C^1(\overline{B_R})$. Suppose $V\in C^2(\overline{B_R})$, $0<S\in C^1(\overline {B_R})$ solve \eqref{eq: curl}, \eqref{eq: div}, \eqref{eq: BC}, and that $S=1$, $\omega=0$ in a neighborhood of $\partial B_R$, then $s=\gamma \log S$, $\rho=V_+^q S^{-1}$ solve \eqref{eq: EP vector} and \eqref{eq: state} on the set $\{\rho>0\}$.
Let $\omega^2\in C^{2,\beta}(\overline{B_R})$ and suppose that $V\in C^{2,\beta}(\overline{B_R})$, $0<S\in C^{1,\beta}(\overline {B_R})$ solve \eqref{eq: curl}, \eqref{eq: div}, \eqref{eq: BC}.   Also suppose that $S=1$, $\omega=0$ in a neighborhood of $\partial B_R$.  Then $s={\gamma} \log S \in C^{1,\beta}(\overline{B_R}) $, $\rho=V_+^q S^{-1}   \in C^{1,\beta}(\overline{B_R})$ solve \eqref{eq: EP vector} and \eqref{eq: state} in the 
set $\{\rho>0\}$.
\end{lemma} 
\begin{proof}
Let 
$$W = S\nabla V-\nabla \left(\frac{1}{|\cdot|}*\frac{V_+^q}S\right)-\kappa\omega^2 r\mathbf e_r$$
on $\overline{B_R}$.
We have
\begin{align}
\nabla\times W &= [S_rV_z-S_zV_r+\kappa(\omega^2 r)_z]\mathbf e_\theta.\label{eq: curl W}\\
\nabla\cdot W &= \nabla \cdot (S\nabla V)+4\pi V_+^{q}S^{-1}-\kappa\nabla\cdot(\omega^2 r \mathbf e_r).\label{eq: div W}
\end{align}
By \eqref{eq: curl W} and \eqref{eq: curl}, $W$ is a $C^{1,\beta} $%$C^1$ 
conservative vector field on $\overline{B_R}$. Thus there exists  $\psi\in C^{2,\beta}(\overline{B_R})$ such that $\nabla \psi = W$. By \eqref{eq: div W} and \eqref{eq: div}, $\Delta \psi = \nabla \cdot W = 0$. So $\psi$ is harmonic on $B_R$. Since $S=1$ and $\omega^2=0$ in a neighborhood of $\partial B_R$, 
\eqn\label{eq: nabla psi}
\nabla \psi = W = \nabla \left(V-\frac{1}{|\cdot|}*\frac{V_+^q}S\right)
\eeqn
in the same neighborhood. By \eqref{eq: nabla psi} and \eqref{eq: BC}, the tangential derivative of $\psi$ on $\partial B_R$ is zero. Thus $\psi$ is a constant on $\partial B_R$. It follows that $\psi$ is identically equal to a constant on $\overline{B_R}$. Thus $W=0$ on $B_R$. In other words, \eqref{eq: EP vector with V} holds on $B_R$. It is now straightforward to see that \eqref{eq: EP vector} and \eqref{eq: state} are equivalent to \eqref{eq: EP vector with V} under the given change of variables where $\rho>0$.
\end{proof}

Our problem is thus reduced to finding solutions of \eqref{eq: curl} and \eqref{eq: div} subject to the boundary data on the floor \eqref{eq: floor s}, \eqref{eq: BC} and the mass constraint \eqref{eq: mass} on the large ball $B_R$. The existence result for our reformulated div-curl system is stated as follows. 

\begin{theorem}\label{thm: main}
Let $s_0\in C^{1,\beta}(\overline{B_R}\cap \{z=0\})$ and $\omega^2\in C^{2,\beta}(\overline{B_R})$ be given. There exist $\epsilon>0$, $\epsilon_1>0$ such that if $|\kappa|+|\mu|<\epsilon_1$, then there exists a unique solution $(V^*,\alpha^*)\in C^{2,\beta}(\overline{B_R})\times \real$, and $S^*\in C^{1,\beta}(\overline{B_R})$ to \eqref{eq: curl}, \eqref{eq: div}, \eqref{eq: floor s}, \eqref{eq: BC}, \eqref{eq: mass}, with $\|(V^*,\alpha^*)-(V^0,\alpha^0)\|_{C^{2,\beta}(\overline{B_R})\times \real}\le\epsilon$. Furthermore, the solution has the properties that $V^*_+$ is supported on $B_{\frac{R_0+R}{2}}$, and $S^*=1$ outside $B_{R_1}$ for some fixed $R_1\in (R_0,R)$.
\end{theorem}

{\it Theorem \ref{thm: intro main} is a direct consequence of Lemma \ref{lem: div curl} and Theorem \ref{thm: main}, as we now show. }

\begin{proof}[Proof of Theorem \ref{thm: intro main}:] We may assume $B=B_R$ where $R>R_0$. By Theorem  \ref{thm: main}, we have a family of solutions $V_{\kappa,\mu}\in C^{2,\beta} (\overline{B_R})$ and $ S_{\kappa,\mu}\in C^{1,\beta}(\overline{B_R})$ which satisfy all the assumptions in Lemma \ref{lem: div curl}. Therefore $s_{\kappa,\mu}= \gamma \log S_{\kappa,\mu}\in C^{1,\beta}(\overline{B_R}) $ and $\rho_{\kappa,\mu}={V_{\kappa,\mu}^q}_+ S_{\kappa,\mu}^{-1} \in C^{1,\beta}(\overline{B_R})$ solve \eqref{eq: EP vector} and \eqref{eq: state} on $\{\rho>0\}$. Together with the velocity profile $\mathbf v= \sqrt{\kappa}~\omega(r,z) (-x_2,x_1,0)$, it is then easy to see that $(\rho_{\kappa,\mu},s_{\kappa,\mu})$ is a time independent solution of \eqref{eq: mass conserv}-\eqref{eq: state} and that $(\rho_{\kappa,\mu},s_{\kappa,\mu})$ satisfies all the properties in Theorem  \ref{thm: intro main}. 
\end{proof}

{\it The rest of the paper is devoted to the proof of Theorem \ref{thm: main}.} Before going any further, we outline our strategy as well as the difficulties in the proof of Theorem \ref{thm: main}.

\subsection{Iteration Scheme and Regularity Difficulties }    \label{iteration subsection} 
%Our problem is thus reduced to finding solutions of \eqref{eq: curl} and \eqref{eq: div} subject to the boundary data on the floor \eqref{eq: floor s}, \eqref{eq: BC} and the mass constraint \eqref{eq: mass} on the large ball $B_R$.  This is done by finding a fixed point of the mapping 
 Our solution to the div-curl system will be constructed as a fixed point of the mapping 
$\F=(\F_1,\F_2)$ defined as follows.  We write 
$$\F_1(V,\al, \ka,\mu)=V^\# \ \text{ and } \ \F_2(V,\al, \ka,\mu)=\al^\#, $$ 
where $V^\#$ is the solution to the linear elliptic equation
\eqn   \label{eq: tilde V elliptic}
\nabla\cdot\left(S\nabla  V^\#\right)=-4\pi V_+^{q}S^{-1}+\kappa \nabla\cdot(\omega^2 r \mathbf e_r)
\eeqn  
on $B_R$ subject to the Dirichlet boundary condition
\eqn\label{eq: tilde V BC}
 V^\#=\frac{1}{|\cdot|}*\frac{V_+^q}S+\alpha \quad %\text{ on }\partial B_R
\eeqn
on $\partial B_R$. Here $S=\mathscr S(V,\kappa,\mu)$ is the solution to the transport equation
\eqn\label{eq: transport}
S_r V_z-S_zV_r = -\kappa( \omega^2r)_z
\eeqn
with the equatorial boundary condition
\eqn\label{eq: IC}
S(r,0)=e^{\frac{\mu s_0(r)}{\gamma}}.
\eeqn
Finally, we define 
\eqn\label{eq: tilde alpha}
\alpha^\# = \alpha +\int V_+^qS^{-1}-M.
\eeqn
The Lane-Emden solution given in Lemma \ref{lem: LE} can now be formulated as $$\F(V^0,\alpha^0,0,0)=(V^0,\alpha^0).$$ 
Our goal  is to look for $(V,\alpha)$ close to $(V^0,\alpha^0)$ solving $\F(V,\alpha,\kappa,\mu)=(V,\alpha)$, for given small parameters $\kappa,\mu$.
We will work in the space $(V,\al)\in C^{2,\beta}(\overline{ B_R}) \times \real$ where $0< \beta<\min(q-1,1)$. 

Most of the difficulty of the problem lies in the subtle loss of regularity of the $V$-component of the mapping $\F$, which is a consequence of the loss of regularity in the $S$ transport equation \eqref{eq: transport}. As we will see in Section \ref{sec: S est}, since the coefficients of the $S$ transport terms in \eqref{eq: transport} involve first derivatives of $V$, $S$ will lose a derivative and belong only to $C^{1,\beta}(\overline{B_R})$ in general. Since at most one derivative of $S$ appears in \eqref{eq: tilde V elliptic}, we will be able to recover $V^\#$ in exactly the same space $C^{2,\beta}(\overline{B_R})$ using Schauder theory. In fact, this is one of the main reasons we chose $S$ and $V$ to be our good variables. By comparison, for instance, if we were to use $s$ and $\rho$ as variables, the corresponding elliptic equation would involve second derivatives of $s$ and could not be used to recover a solution in the same regularity space. 

However, even with the choice of our good variables, $\F$ still suffers a loss of regularity at the level of its Fr\'echet derivative.  Of course, the natural method to solve the perturbed fixed point problem is to apply the implicit function theorem. Such an approach requires the existence of the Fr\'echet differential in the same regularity space. If one then tries to vary $V$ by a variation $\delta V$, the corresponding variation $\delta S$ satisfies
$$
(\delta S)_rV_z-(\delta S)_zV_r= -(S+\delta S)_r(\delta V)_z+(S+\delta S)_z(\delta V)_r,
$$
which linearizes to 
\eqn\label{eq: F Frechet}
(\delta S)_rV_z-(\delta S)_zV_r= -S_r(\delta V)_z+S_z(\delta V)_r.
\eeqn
The right hand side of \eqref{eq: F Frechet} involves first derivatives of $S$, which in general belong only to $C^{0,\beta}(\overline{B_R})$. As a consequence, one can only obtain $C^{0,\beta}(\overline{B_R})$ estimates on $\delta S$.  
$\delta S$ will be smoother along the characteristics, but not so along the transverse direction.   This is not enough to recover the Fr\'echet differential of $\F$ in $C^{2,\beta}(\overline{B_R})$. This loss of derivative is an essential feature of the entropy transport equation and cannot be remedied by discovering any hidden structure in the problem. 
In fact, $\F$ does not even depend continuously on $V$ in the space $C^{2,\beta}(\overline{B_R})$ as soon as one shifts away from the special Lane-Emden solution $V^0$. The essential reason is that the $C^\beta$ H\"older norm is discontinuous with respect to smooth inner variations of a $C^\beta$ function, such as a horizontal shift: $f(x)\mapsto f(x+h)$.  As a result of \eqref{eq: transport}, the characteristics along which $S$ is transported are level curves of $V$. Thus a smooth variation of $V$ can cause these characteristics to be shifted horizontally, resulting in a discontinuous variation of $S$ in $C^{1,\beta}(\overline{B_R})$. Since the first derivative of $S$ appears in \eqref{eq: tilde V elliptic}, the discontinuity would be passed on to $V^\#$.

We overcome this difficulty by working directly with the standard Newton iteration scheme (as in the proof of the implicit function theorem). For simplicity, we denote by $D\F^0$ the Fr\'echet derivative of $\F$ with respect to $(V,\alpha)$ at $(V^0,\alpha^0,0,0)$.  
Recall that the entropy is constant ($S\equiv 1$) for the Lane-Emden solution. Thus the above mentioned loss of regularity on the Fr\'echet differential does not happen at this particular point in function space. We write the equation $(V,\alpha)=\F(V,\alpha,\kappa,\mu)$ as 
$$
(V,\alpha)-(V^0,\alpha^0)=\F(V,\alpha,\kappa,\mu)-\F(V^0,\alpha^0,0,0),
$$
and further as
\begin{align*}
&~(V-V^0,\alpha-\alpha^0)-D\F^0(V-V^0,\alpha-\alpha^0)\\
=&~\F(V,\alpha,\kappa,\mu)-\F(V^0,\alpha^0,0,0)-D\F^0(V-V^0,\alpha-\alpha^0).
\end{align*}
Inverting $I-D\F^0$, we obtain 
\begin{align}
&~(V-V^0,\alpha-\alpha^0)\notag\\
=&~(I-D\F^0)^{-1}[\F(V,\alpha,\kappa,\mu)-\F(V^0,\alpha^0,0,0)-D\F^0(V-V^0,\alpha-\alpha^0)]. \label{eq: fixed point 2}
\end{align}
As is explained above, we will not attempt to show that the right hand side of \eqref{eq: fixed point 2} is a contraction in  $C^{2,\beta}(\overline{B_R})\times\real$ (as is commonly done in the proof of the implicit function theorem). Instead, we produce a sequence $(V_n,\alpha_n)$ of approximate solutions by iterating the right hand side of \eqref{eq: fixed point 2}, and show that: 
	\begin{enumerate}[(i)]
\item The sequence is uniformly bounded and remains near $(V^0,\alpha^0)$ in the space $C^{2,\beta}(\overline{B_R})\times \real$.

\item $(V_{n+1}-V_n,\alpha_{n+1}-\alpha_n)$ contracts in the weaker space $C^{1,\beta}(\overline{B_R})\times \real$ as $n$ grows.
	\end{enumerate}
Note that the loss of derivative of $\delta S$ in \eqref{eq: F Frechet} described above can now be tolerated if the goal is to only get the weaker $C^{1,\beta}(\overline{B_R})$ estimates on $V_{n+1}-V_n$.
We will then combine the smooth norm estimates with the rough ones using interpolation of H\"older norms, and conclude that the iteration sequence converges in $C^{2,\beta'}(\overline{B_R})\times \real$ for some $0<\beta'<\beta$, which will be enough to get a solution to our problem.

The whole problem is of quasilinear type, because $S$ has the same regularity as $\nabla V$, which then appears as the coefficient of the second order elliptic term in \eqref{eq: tilde V elliptic}. Of course, first derivatives of $S$ also appear in \eqref{eq: tilde V elliptic}, which seem like fully nonlinear combinations of second derivatives of $V$. However, when the solutions are close to Lane-Emden, the first derivatives of $S$ are small in size, and can be absorbed by the main elliptic term. The method of combining estimates of smooth norms with rough norms may also prove useful in treating other quasilinear hyperbolic problems. 
%Even though $\F$ is defined in a small neighborhood of the Lane-Emden solution 
%$(V^0,0)$, it is not a contraction on this neighborhood, even when the parameters $\mu,\ka$ 
%are taken to be small.  The reason is as follows.  When \eqref{eq: transport} is solved for $S$, 
%we only get $S$ to be in $C^{1,\beta}$.  
%Then $S$ is plugged into \eqref{eq: tilde V elliptic} and the Schauder estimate yields 
%$V$ to be in $C^{2,\beta}$.  However, if one asks whether $\F$ is a contraction, 
%it is required to take differences like $\F(V_1,\al_1)-\F(V_1,\al_1)$, which is similar to  
%an extra derivative being required on $S$ and is not possible.  
%So we alternatively find the fixed point of $\F$ by doing a standard iteration 
%to get an approximate sequence $\{V_n\}$, then 
%obtain a bound on $V_n$ in the $C^{2,\beta}$ norm, 
%and finally prove the contraction property in the weaker $C^{1,\beta'}$ norm.  
%
%The first step in the proof is to solve \eqref{eq: transport} for $S=\S(V)$ as a function of $V$.  
%Of course, this is simply accomplished by means of the characteristics.  
%In order to obtain the required mapping properties on H\"older spaces, 
%we explicitly introduce coordinates for the characteristic curves.  
%{\color{blue} MORE DETAIL and explain informally our approach and why the naive implicit function theorem fails. set the road map: section 3 does what, section 4 does what...}

\subsection{Outline of Paper}

The rest of the paper will be organized as follows. In Section \ref{sec: S est}, we obtain the $C^{1,\beta}(\overline{B_R})$ estimates on $S$ by solving \eqref{eq: transport} using the method of characteristics. Although the solution is easy to write down formally, one cannot automatically get the required estimates, since the characteristic vector field is degenerate near the origin.   In order to properly estimate $S$ near the origin, we must prove certain weighted H\"older estimates on the characteristic coordinates. This makes Section  \ref{sec: S est} rather technical. 
In Section \ref{sec: iteration}, we prove the invertibility of the linearized operator at the Lane-Emden solution and then construct the iteration map using this inverse. In Section \ref{sec: convergence}, we prove convergence of the iteration scheme in $C^{2,\beta'}(\overline{B_R})\times\real$ for any $0<\beta'<\beta$, by first showing uniform boundedness of the smooth norms, and then showing contraction of the rough norms of the iteration sequence. We also show eventually that the solution thus obtained actually lies in the smoother space $C^{2,\beta}(\overline{B_R})$, and is unique.
%%%%%%%%%%%%%%%%%%%%%%%%%%%%

\section{Entropy Estimates}\label{sec: S est}

In this section, we derive the $C^{1,\beta}(\overline{B_R})$ estimates for $S$ via the method of characteristics. It is immediate from \eqref{eq: transport} that the characteristics are the level curves of $V$. A minor problem arises when one tries to solve \eqref{eq: transport} with boundary condition \eqref{eq: IC}. When $V$ is perturbed away from $V^0$, some of the characteristics near $\partial B_R$ may end on $\partial B_R$ and not on $\{z=0\}$. In this case, the boundary data \eqref{eq: IC} on the equatorial plane is not enough to uniquely determine $S$. To avoid such ambiguity in the definition of $S$, we solve \eqref{eq: transport} in a slightly larger ball.
%\subsection{Extension and key lemma}
% In particular, let $V^0$ be the Lane-Emden solution,

First of all, as in Lemma \ref{lem: LE}, we regard $V^0$ as defined not only on $B_R$, but on the entire space $\real^3$. 
A function $V\in C^{2,\beta}(\overline{B_R})$ defined on $B_R$ can be extended  to $B_{R+1}$ in the following way. Let $$E: C^{2,\beta}(\overline{B_R})\to C^{2,\beta}(\overline{B_{R+1}})$$ be a bounded linear extension operator such that the support of $Eu$ is contained in $B_{R+\frac12}$ for all $u\in C^{2,\beta}(\overline{B_R})$. 
We  extend $V$ by $$V_{\text{ext}}=V^0+E(V-V^0).$$ Such an extended $V$ has the property that if $\|V-V^0\|_{C^{2,\beta}(\overline{B_R})}<\delta$, then $$ \|V_{\text{ext}}-V^0\|_{C^{2,\beta}(\overline{B_{R+1}})}<2\delta 
\ \text{ and } \ V_{\text{ext}}\equiv V^0\text{ on }\overline{B_{R+1}}\setminus B_{R+\frac12}.$$  %$V_{\text{extended}}$ is close to $V^0$ in $C^{2,\beta}(\overline{B_{R+1}})$ whenever $V$ is close to $V^0$ on $C^{2,\beta}(\overline{B_R})$, and $V_{\text{extended}}\equiv V^0$ outside $B_{R+\frac12}$. 
Slightly abusing the notation, we still denote $V_{\text{ext}}$ by $V$. As is explained in the discussion following Lemma \ref{lem: support}, $s_0$ and $\omega^2$ can be regarded as defined on $B_{R+1}$.
Then we  define $S$ to be the solution to \eqref{eq: transport} and \eqref{eq: IC} on $B_{R+1}$. Since $V=V^0$ outside $B_{R+\frac12}$, none of the characteristics will penetrate $\partial B_{R+1}$, and $S$ is uniquely defined.

We are now ready to state the main result of this section.

\begin{theorem}    \label{lem: V to S map}
Fix $\beta\in (0,1)$. Let $\omega^2$ and $s_0$ be prepared as above.  Assume $s_0\in C^{1,\beta}(\overline{B_{R+1}}\cap \{z=0\})$ and $\omega^2\in C^{2,\beta}(\overline{B_{R+1}})$. 

(i)  Then there exists  $\delta>0$ such that for every $V$ with $\|V-V^0\|_{C^{2,\beta}(\overline{B_R})}<\delta$, and extended to $B_{R+1}$ as explained above, the equations \eqref{eq: transport}, \eqref{eq: IC} admit a unique solution $S\in C^{1,\beta}(\overline{B_{R+1}})$. Moreover, there is a radius  $R_1\in (R^0, R)$ such that $S(r,z)\equiv 1$ for $|(r,z)|>R_1$. 

%Let $B_\delta(V^0)$ be given as in Lemma \ref{lem: V to S map}. 
(ii) If $|\kappa|+|\mu|<1$, then there exists a constant $C>0$ such that the function $S$ %constructed in Lemma \ref{lem: V to S map} 
satisfies 
\eqn \label{S near 1}
\|S-1\|_{C^{1,\beta}(\overline{B_{R+1}})}<C(|\kappa|+|\mu|).\eeqn
\end{theorem}

In order to prove Theorem \ref{lem: V to S map}, 
we will of course construct the solution $S$ by the method of characteristics. The characteristic ODEs associated with the transport equation \eqref{eq: transport} are given by 
\begin{equation}\label{eq: ODE}
\begin{cases}
\frac{\partial r}{\partial t}(t,\tau)=V_z(r(t,\tau),z(t,\tau)), \\
\frac{\partial z}{\partial t}(t,\tau)=-V_r(r(t,\tau),z(t,\tau)),\\
r(0,\tau)= \tau,~z(0,\tau)=0.
\end{cases}
\end{equation} 
We are using $t$ as a parameter along the characteristic curves  and $\tau$ as a label of the individual curves.  ($t$ has nothing to do with time.) 
These characteristic equations define a mapping $\phi : (t,\tau)\to (r,z)$.  

The remainder of Section \ref{sec: S est} constitutes the proof of Theorem \ref{lem: V to S map}. It involves technical estimates of $S$ by using the characteristic coordinates $(t,\tau)$. They may be skipped on first reading if the reader is only interested in the main existence proof. 

%%%%%%%%%%%%%%%%%%%%%%%%%%%%%%%%
\subsection{Characteristic Coordinates}\label{sec: char coord}

We now construct the characteristic coordinates on 
\eqn
\Omega = B_{R+1}\cap \{r>0,z>0\}, 
\eeqn
which amounts to the construction of the diffeomorphism which maps a proper open set in $(t,\tau)$ coordinates onto $\Omega$. We will consistently use $\|\cdot\|_k$ and $\|\cdot\|_{k,\beta}$ to denote the $C^{k}(\overline{B_R})$ and $C^{k,\beta}(\overline{B_R})$ norms.
We begin by constructing the mapping $\phi$ on a precise set.

\begin{prop}\label{lem: char coord}
There exists $\delta>0$ sufficiently small such that, for any function $V$ for which $\|V-V^0\|_2<\delta$, there exists an open subset $O\subset \{(t,\tau)~|~t>0,0<\tau<R+1\}$ and a $C^1$ diffeomorphism $\tilde{\phi}:O\to \Omega$ such that $(r(t,\tau),z(t,\tau))=\tilde\phi(t,\tau)$ extends to a $C^1$ map on $\overline O\setminus \{\tau=0\}$, and solves \eqref{eq: ODE}. 
\end{prop}

We will prove this proposition by means of a few lemmas. We first write down the trivial coordinates $(r^0,z^0)$ induced by the nonrotating $V^0$. 
%When $V=V^0(\sqrt{r^2+z^2})$, we denote the solution to \eqref{eq: ODE} by $(r^0(t,\tau),z^0(t,\tau))=\phi^0(t,\tau)$. 

\begin{lemma}\label{lem: phi0} Let $(r^0(t,\tau),z^0(t,\tau))=\phi^0(t,\tau)$ be the solution to  \eqref{eq: ODE} for the function $V=V^0(\sqrt{r^2+z^2})$. Then 
\eqn\label{eq: phi 0}
(r^0(t,\tau),z^0(t,\tau)) = \tau \left(\cos \left(-\frac{(V^0)'(\tau) t}{\tau}\right),\sin\left( -\frac{(V^0)'(\tau) t}{\tau}\right)\right)
\eeqn
with its corresponding domain $$O^0=\bigg\{(t,\tau), 0<\tau<R+1, 0<t<\frac{\pi  \tau}{-2(V^0)'(\tau)}\bigg\}.$$
In particular, $|\phi^0(t,\tau)|=\tau$. 
\end{lemma}

\begin{proof} Since $V^0(r,z)=V^0(\sqrt{r^2+z^2})$, %$V^0_r= \frac{r}{\sqrt{r^2+z^2}}(V^0)'$ and $V^0_z= \frac{z}{\sqrt{r^2+z^2}}(V^0)'$. Therefore, 
we can easily check that \eqref{eq: ODE} leads to $\frac{\partial}{\partial t} \left[(r^0)^2 + (z^0)^2\right]=0$, which in turn implies $(r^0)^2 + (z^0)^2 = \tau^2$. Hence $r^0(t,\tau)=\tau \cos \theta(t)$ and $z^0(t,\tau)= \tau \sin \theta(t)$ for some $\theta(t)$. The first equation of \eqref{eq: ODE} gives $\frac{d\theta}{dt}= - \frac{(V^0)'(\tau)}{\tau}$. Since $r^0(0,\tau)=\tau$, $\theta(0)=0$. Thus $\theta(t)= - \frac{(V^0)'(\tau)}{\tau} t$. 
In order to cover the first quadrant in $r^0,z^0$ coordinates, we want $0<\theta(t)<\frac{\pi}{2}$. This verifies the domain $O^0$ specified above. The last assertion trivially follows from $(r^0)^2 + (z^0)^2 = \tau^2$. 
\end{proof}

%One can easily get 

%The corresponding domain $O^0=\bigg\{(t,\tau), 0<\tau<R+1, 0<t<\frac{\pi  \tau}{-2(V^0)'(\tau)}\bigg\}$. 

For general $V$, the domain will be distorted from $O^0$.  For convenience, we will actually construct the map $\phi$ to a slightly extended domain $O^a$ 
such that $\phi(t,\tau)=(r(t,\tau),z(t,\tau))$ solves \eqref{eq: ODE}. To this end, we define 
 $$O^a =\bigg\{(t,\tau), 0<\tau<R+1, -\frac{\pi}{4a}<t<\frac{\pi  \tau}{-2(V^0)'(\tau)}+\frac{\pi}{4a}\bigg\}$$
where $$a = \sup_{0<\tau<R+1}\frac{-(V^0)'(\tau)}{\tau}>0$$
by Lemma \ref{lem: LE}.
%\bcr
%Here remind that we allow the solution $\phi=(r,z)$ of \eqref{eq: ODE} to be located in the second and fourth quadrants 
%and recall the axi-symmetry of $V$ ? 
%\ec

\begin{lemma}\label{lem: phi LM separation}
Define $\phi$ to be the solution to \eqref{eq: ODE}. Then $\phi$ is a $C^1$ map on $O^a$, with a range avoiding the third quadrant, and there exists a constant $C>0$ such that for $(t,\tau), (t_1,\tau), (t_2, \tau), (t,\tau_1),(t,\tau_2)\in O^a$, 
\eqn\label{eq: char coord deviation}
|\phi(t,\tau)-\phi^0(t,\tau)|\le C\|V-V^0\|_2 t\tau,%\phi^0(t,\tau)|,
\eeqn
\begin{align}\label{eq: char coord separation}
|\phi(t_1,\tau)-\phi^0(t_1,\tau)-\phi(t_2,\tau)&+\phi^0(t_2,\tau)|\notag\\
&\le C\|V-V^0\|_2|\phi^0(t_1,\tau)-\phi^0(t_2,\tau)|,
\end{align}
\begin{align}\label{eq: char coord tau separation}
|\phi(t,\tau_1)-\phi^0(t,\tau_1)-\phi(t,\tau_2)&+\phi^0(t,\tau_2)|\notag\\
&\le C\|V-V^0\|_{2}|\phi^0(t,\tau_1)-\phi^0(t,\tau_2)|.
\end{align}
\end{lemma}

\begin{remark}\label{rmk: char coord 1}
Lemma \ref{lem: phi LM separation} concerns solutions to \eqref{eq: ODE} starting at $(\tau, 0)$, but similar estimates hold for solutions starting at a general point $(r,z)$ in $O^a$.
\end{remark}
\begin{remark}
An obvious consequence of \eqref{eq: char coord deviation} is that
\eqn\label{eq: phi close to tau}
\frac1{1+C\delta}\tau\le |\phi(t,\tau)|\le (1+C\delta)\tau
\eeqn
when $\|V-V^0\|_{2}<\delta$.
 \end{remark}
 
\begin{proof}
We begin by noting that the solution to \eqref{eq: ODE} is defined for all $t$ if $0<\tau< R+1$, because the characteristics are confined to a compact region $\overline{B_{R+1}}$. $\phi$ is a $C^1$ map as $\nabla V\in C^1$. We denote $\phi^0=(r^0,z^0)$. 

\noindent{\bf Proof of \eqref{eq: char coord deviation}.} Writing 
\begin{align*}
\frac{\partial r}{\partial t}(t,\tau)-\frac{\partial{r^0}}{\partial t}(t,\tau)&=V_z(r,z)-(V^0)_z(r^0,z^0)\\
&=V_z(r,z)-V_z(r^0,z^0)+(V-V^0)_z(r^0,z^0),
\end{align*}
\begin{align*}
\frac{\partial z}{\partial t}(t,\tau)-\frac{\partial{z^0}}{\partial t}(t,\tau)&=V_r(r,z)-(V^0)_r(r^0,z^0)\\
&=V_r(r,z)-V_r(r^0,z^0)+(V-V^0)_r(r^0,z^0),
\end{align*}
and $\nabla (V-V^0)(r^0,z^0)=\nabla (V-V^0)(r^0,z^0)-\nabla (V-V^0)(0,0)$,
we can integrate and estimate
\begin{align}
|\phi(t,\tau)-\phi^0(t,\tau)|&\le 2\|V\|_{2}\int_0^t |\phi(s,\tau)-\phi^0(s,\tau)|~dt\notag\\
&\qquad +2\|V-V^0\|_{2}\int_0^t|\phi^0(s,\tau)|~ds.
\end{align}
By Gronwall's inequality, this implies
\begin{align}
|\phi(t,\tau)-\phi^0(t,\tau)|\le 2e^{2\|V\|_{2}t}\|V-V^0\|_{2}\tau t \le C\|V-V^0\|_{2}\tau t.
\end{align}
Here we have used $|\phi^0(s,\tau)|=\tau$ as shown in Lemma \ref{lem: phi0}. %Here we have used \eqref{eq: phi 0} to get $|\phi^0(s,\tau)|=\tau$. 
%Taking the limit as $\epsilon\to0$, we get \eqref{eq: char coord deviation}. 

\noindent{\bf Proof of \eqref{eq: char coord separation}.}  We write
\begin{align*}
&~\frac{\partial}{\partial t}\left(r(t,\tau)-r^0(t,\tau)-r(t_1,\tau)+r^0(t_1,\tau)\right)\\
=&~V_z(r,z)-(V^0)_z(r^0,z^0)\\
=&~V_z(r,z)-V_z(r^0,z^0)+(V-V^0)_z(r^0,z^0),
\end{align*}
\begin{align*}
&~\frac{\partial}{\partial t}\left(z(t,\tau)-z^0(t,\tau)-z(t_1,\tau)+z^0(t_1,\tau)\right)\\
=&~V_r(r,z)-(V^0)_r(r^0,z^0)\\
=&~V_r(r,z)-V_r(r^0,z^0)+(V-V^0)_r(r^0,z^0),
\end{align*}
and use
\begin{align*}
|\nabla V(r,z)-\nabla V(r^0,z^0)|&\le \|V\|_{2}|\phi(t,\tau)-\phi^0(t,\tau)|\\
&\le \|V\|_{2}|\phi(t,\tau)-\phi^0(t,\tau)-\phi(t_1,\tau)+\phi^0(t_1,\tau)|\\
&\quad +\|V\|_{2}|\phi(t_1,\tau)-\phi^0(t_1,\tau)|\\
&\le \|V\|_{2}|\phi(t,\tau)-\phi^0(t,\tau)-\phi(t_1,\tau)+\phi^0(t_1,\tau)|\\
&\quad +C\|V-V^0\|_{2}\tau.
\end{align*}
Here we have used \eqref{eq: char coord deviation} on the last term. As a consequence,
\begin{align*}
&~|\phi(t,\tau)-\phi^0(t,\tau)-\phi(t_1,\tau)+\phi^0(t_1,\tau)|\\
\le&~ C\|V\|_{2}\int_{t_1}^t |\phi(s,\tau)-\phi^0(s,\tau)-\phi(t_1,\tau)+\phi^0(t_1,\tau)|~ds\\
&\quad +C\|V-V^0\|_{2}\tau (t-t_1).
\end{align*}
Now we use Gronwall on $[t_1,t_2]$ to get
\begin{align}
&~|\phi(t_1,\tau)-\phi^0(t_1,\tau)-\phi(t_2,\tau)+\phi^0(t_2,\tau)|\notag\\
\le &~C\|V-V^0\|_{2}\tau (t_2-t_1).
\end{align}
Finally, observe that the range of $t_1,t_2$ and \eqref{eq: phi 0} imply 
\eqn
\tau(t_2-t_1)\le C|\phi^0(t_1,\tau)-\phi^0(t_2,\tau)|,
\eeqn
and the proof of \eqref{eq: char coord separation} is complete.

\noindent{\bf Proof of \eqref{eq: char coord tau separation}.}  To prove \eqref{eq: char coord tau separation}, we consider the equations satisfied by $\phi_\tau$:
\begin{equation}\label{eq: var ODE}
\begin{cases}
\frac{\partial r_\tau}{\partial t}=V_{zr}(r,z)r_\tau+V_{zz}(r,z)z_\tau, \\
\frac{\partial z_\tau}{\partial t}=-V_{rr}(r,z)r_\tau-V_{rz}(r,z)z_\tau,\\
r_\tau(0,\tau)= 1,~z_\tau(0,\tau)=0,
\end{cases}
\end{equation}
and similar equations satisfied by $\phi^0_\tau$ with $V^0$ replacing $V$. Taking differences of the $\phi_\tau$ equations and $\phi^0_\tau$ equations as before, we get
\begin{align}
\left|\frac{\partial}{\partial t}\phi_\tau(t,\tau)-\phi^0_\tau(t,\tau)\right|&\le C\|V\|_{2}|\phi_\tau-\phi^0_\tau|+C\|V-V^0\|_{2}\notag\\
&\quad +C\|V^0\|_{3}|\phi-\phi^0|\notag\\
&\le C\|V\|_{2}|\phi_\tau-\phi^0_\tau|+C\|V-V^0\|_{2}.
\end{align}
Here we have used \eqref{eq: char coord deviation} to absorb the last term by the second term.   We integrate and use Gronwall as before to conclude
\eqn\label{eq: phi_tau}
|\phi_\tau(t,\tau)-\phi^0_\tau(t,\tau)|\le C\|V-V^0\|_{2}.
\eeqn
We integrate in $\tau$ on $[\tau_1,\tau_2]$ to get
$$|\phi(t,\tau_1)-\phi^0(t,\tau_1)-\phi(t,\tau_2)+\phi^0(t,\tau_2)|\le C\|V-V^0\|_{2}|\tau_1-\tau_2|.$$
Finally observing 
$$|\tau_1-\tau_2|\le |\phi^0(t,\tau_1)-\phi^0(t,\tau_2)|$$
completes the proof.
\end{proof}

%We are now ready to show Proposition  \ref{lem: char coord}. 

\begin{proof}[\bf Proof of Proposition \ref{lem: char coord}]
%We prove this Lemma in several steps: 
First we will show that $\phi$ is one-to-one on $O^a$; secondly we will show that $\phi(O^a)$ covers $\Omega$ %the disc $B_{R+1}$ in the first quadrant
; and finally we will show that $\phi$ has non-zero Jacobian on $O^a$. Then we will just define $\tilde\phi$ to be the restriction of $\phi$ to $\phi^{-1}(\Omega)$.

\medskip

\noindent{\it{\underline{Step 1: $\phi$ is 1-1 on $O^a$.}}}  We begin by observing that $\phi(t_1,\tau_1)\ne \phi(t_2,\tau_2)$ whenever $\tau_1\ne\tau_2$. Suppose this were false  and $t_1\ge t_2$. Then $\phi(t_1,\tau_1)= \phi(t_2,\tau_2)$, $\tau_1\ne\tau_2$ and  $t_1\ge t_2$. Tracing the solution backwards in $t$, we would have $\phi(t_1-t_2,\tau_1)=\phi(0,\tau_2)=(\tau_2,0)$  for $t_1-t_2\ge 0$. This is not possible, because the solution never enters the third quadrant of the $(r,z)$ plane and can only go from the fourth quadrant to the first quadrant by the positivity of  $-V_r$ on the positive $r$-axis. To see the latter fact, write $V_r=(V^0)_r+(V-V^0)_r$ and observe that $\inf_{0<r<R+1}\frac{-(V^0)'(r)}{r}>0$ and
\eqn
|(V-V^0)_r(r,0)|=|(V-V^0)_r(r,0)-(V-V^0)_r(0,0)|\le \|V-V^0\|_{2}\cdot r.
\eeqn
Thus 
\eqn\label{eq: V_r<0}
-V_r(r,0)>0 \quad \text{ for }0<r<R+1
\eeqn
if $\|V-V^0\|_{2}$ is sufficiently small. 

Now we consider $\phi(t_1,\tau)$ and $\phi(t_2,\tau)$ with $t_1\ne t_2$. We obviously have $\phi^0(t_1,\tau)\ne \phi^0(t_2,\tau)$.  Inequality \eqref{eq: char coord separation} now implies $\phi(t_1,\tau)\ne \phi(t_2,\tau)$.

\medskip

\noindent{\it{\underline{Step 2: $\Omega \subset \phi (O^a)$.}}}  We pick $(r_0,z_0)\in \Omega$. Let $\tau_0=\sqrt{r_0^2+z_0^2}$. Denote the polar coordinates of $(r_0,z_0)$ by $(\tau_0,\theta_0)$. We solve the ODE in \eqref{eq: ODE} backwards in $t$ for an interval of length $t_0=\frac{\theta\tau_0}{-(V^0)'(\tau_0)}+\frac{\pi}{8a}$, with initial value at $(\tau_0,\theta_0)$. If $V=V^0$, the solution reaches the final point with polar coordinates $(\tau_0,\theta_{\text{f}})=(\tau_0, -\frac{-(V^0)'(\tau_0)}{\tau_0}\frac{\pi}{8a})$. Note that since $\inf_{0<\tau<R+1}\frac{-(V^0)'(\tau)}{\tau}=b> 0$, we have $|\theta_{\text f}|>\frac{\pi b}{8a}$. We now use Remark \ref{rmk: char coord 1} and use a modified version of \eqref{eq: char coord deviation} with starting point at $(\tau_0,\theta_0)$ to conclude that for a general $V$, we have 
\eqn
|(r,z)(-t_0)-\tau_0(\cos\theta_{\text f},\sin\theta_{\text f})|\le C\|V-V^0\|_{2}\cdot\tau_0.
\eeqn
Together with the lower bound on $|\theta_{\text f }|$, this implies that $(r,z)(-t_0)$ is in the fourth quadrant, as long as $\|V-V^0\|_{2}$ is sufficiently small. Since the only way to enter the fourth quadrant is to pass the positive $r$-axis, we conclude that the solution $(r,z)(t)$ must intersect the segment $0<r<R+1$ at some value $r=\tau_1$ for some $t=-t_1\in (-t_0,0)$. By uniqueness, $\phi(t_1,\tau_1)=(r_0,z_0)$. It remains to show that $(t_1,\tau_1)\in O^a$. In fact, we use Remark \ref{rmk: char coord 1} again to get 
\eqn
|(\tau_1,0)-\tau_0(\cos \theta_1,\sin\theta_1)|\le C\|V-V^0\|_{2}\cdot\tau_0,
\eeqn
where $\theta_1$ is the polar angle of $(r,z)(-t_1)$ when $V=V^0$.
This implies 
\eqn
|\tau_0-\tau_1|\le C\|V-V^0\|_{2}.
\eeqn
By the smoothness of $V^0$, we have 
\eqn
\bigg|\frac{\tau_0}{(V^0)'(\tau_0)}-\frac{\tau_1}{(V^0)'(\tau_1)}\bigg|\le C|\tau_0-\tau_1|\le C\|V-V^0\|_{2}.
\eeqn
We now note that 
\begin{align}
0<t_1<t_0<\frac{\pi \tau}{-2(V^0)'(\tau)}+\frac\pi {8a}&\le C\|V-V^0\|_{2}+\frac{\pi \tau_1}{-2(V^0)'(\tau_1)}+\frac\pi {8a}\notag\\
&\le \frac{\pi \tau_1}{-2(V^0)'(\tau_1)}+\frac\pi {4a},
\end{align}
provided $\|V-V^0\|_{2}$ is sufficiently small. Thus $(t_1,\tau_1)\in O^a$.

\medskip

\noindent{\it{\underline{Step 3: $\phi$ has non-zero Jacobian on $O^a$. }}}  We now compute the Jacobian of $\phi(t,\tau)=(r(t,\tau),z(t,\tau))$. 
\begin{align}
r_tz_\tau-r_\tau z_t&=V_z(r(t,\tau),z(t,\tau))z_\tau+V_r(r(t,\tau),z(t,\tau))r_\tau\notag\\
&=\frac{d}{d\tau}V(r(t,\tau),z(t,\tau)).
\end{align}
On the other hand,  
\eqn
\frac{d}{dt}V(r(t,\tau),z(t,\tau))=V_rr_t+V_zz_t = V_rV_z-V_zV_r=0, 
\eeqn
so that 
\eqn
V(r(t,\tau),z(t,\tau))=V(r(0,\tau),z(0,\tau))=V(\tau,0).
\eeqn
Thus
\eqn\label{eq: Jacobian}
r_tz_\tau-r_\tau z_t=\frac{d}{d\tau}V(r(t,\tau),z(t,\tau))=V_r(\tau,0)<0
\eeqn
by \eqref{eq: V_r<0}.
We finish the proof by defining $O=\phi^{-1}(\Omega)$, and $\tilde\phi = \phi\big|_O$.
\end{proof}

Our next task is to derive some H\"older type estimates of the first derivatives of the characteristic coordinates.

\subsection{Regularity of the Characteristics}

\begin{lemma}
Fix $\beta\in (0,1)$. Let $\|V-V^0\|_{2,\beta}$ be sufficiently small and $\phi:O\to \Omega$ be the diffeomorphism given in Proposition \ref{lem: char coord} and let $(t,\tau), (t_1,\tau_1),(t_2,\tau_2)\in O$, $(r,z),(r_1,z_1),(r_2,z_2)\in\Omega$. Denote $\phi^{-1}(r,z)=(t(r,z),\tau(r,z))$. There exists a constant $C>0$ such that
\eqn\label{eq: grad phi 1}
|\phi_t(t,\tau)|\le C\tau,~|\phi_\tau(t,\tau)|\le C,
\eeqn
\eqn\label{eq: grad phi 2}
 |\nabla t(r,z)|\le \frac{C}{\sqrt{r^2+z^2}},~|\nabla \tau(r,z)|\le C.
\eeqn
\eqn\label{eq: grad phi 4}
|\phi_\tau(t_1,\tau_1)-\phi_\tau(t_2,\tau_2)|\le C(|t_1-t_2|+|\tau_1-\tau_2|^\beta),
\eeqn
\eqn\label{eq: grad phi 3}
\left|\frac1{\tau_1}\phi_t(t_1,\tau_1)-\frac{1}{\tau_2}\phi_t(t_2,\tau_2)\right|\le C(|t_1-t_2|^\beta+|\tau_1-\tau_2|^\beta),
\eeqn
\eqn\label{eq: grad phi 5}
|(r_1^2+z_1^2)^{\frac{1+\beta}{2}}\nabla t(r_1,z_1)-(r_2^2+z_2^2)^{\frac{1+\beta}{2}}\nabla t(r_2,z_2)|\le C|(r_1,z_1)-(r_2,z_2)|^\beta,
\eeqn
\eqn\label{eq: grad phi 6}
\left|(r_1^2+z_1^2)^{\frac\beta2}\nabla \tau(r_1,z_1)-(r_2^2+z_2^2)^{\frac\beta2}\nabla \tau(r_2,z_2)\right|\le C|(r_1,z_1)-(r_2,z_2)|^\beta
\eeqn

\end{lemma}

\begin{proof}  \noindent{\bf Proof of \eqref{eq: grad phi 1}.}  
By \eqref{eq: ODE} and \eqref{eq: phi close to tau}, 
$$|\phi_t(t,\tau)|\le \|V\|_{2}|\phi(t,\tau)|\le C\|V\|_{2}\tau.$$
The bound on $\phi_\tau$ is easily obtained from \eqref{eq: var ODE}. In fact, \eqref{eq: phi_tau} gives such a bound.

\noindent{\bf Proof of \eqref{eq: grad phi 2}.}  To estimate $\nabla \phi^{-1}$, we note that
\eqn
t_r=\frac{z_\tau}{r_tz_\tau-r_\tau z_t}, ~t_z=\frac{-r_\tau}{r_tz_\tau-r_\tau z_t},
\eeqn
\eqn\label{eq: grad tau est 1}
\tau_r=\frac{-z_t}{r_tz_\tau-r_\tau z_t},~\tau_z=\frac{r_t}{r_tz_\tau-r_\tau z_t}.
\eeqn
Let us estimate $t_r$. First recall \eqref{eq: Jacobian} and write
\eqn\label{eq: t_r est 1}
t_r=\frac{z_\tau}{\frac1\tau V_r(\tau,0)}\frac1\tau.
\eeqn
Now $|z_\tau|\le C$ by \eqref{eq: grad phi 1}, and since $(V_r-V_r^0)(0,0)=0$,
\begin{align*}
\frac1\tau V_r(\tau,0)%&=\int_0^1 V_{rr}(s\tau,0)~ds=\int_0^1 (V^0)_{rr}(s\tau,0)~ds+\int_0^1 (V-V^0)_{rr}(s\tau,0)~ds\\
&=\frac{(V^0)_r(\tau,0)}{\tau}+\int_0^1 (V-V^0)_{rr}(s\tau,0)~ds.
\end{align*} 
Thus
\eqn
\left|\frac{V_r(\tau,0)}\tau-\frac{(V^0)_r(\tau,0)}{\tau}\right|\le \|V-V^0\|_{2}
\eeqn
Since $\inf_{0<\tau<R+1}\frac{-(V^0)_r(\tau,0)}{\tau}=b>0$, we have $\left|\frac{V_r(\tau,0)}\tau\right|>\frac b2>0$ when $\|V-V^0\|_{2}$ is sufficiently small. We deduce from \eqref{eq: t_r est 1}, \eqref{eq: grad phi 1} and \eqref{eq: phi close to tau} that 
\eqn
|t_r|\le \frac C\tau \le \frac{C}{|\phi(t,\tau)|}=\frac{C}{\sqrt{r^2+z^2}}.
\eeqn
The expressions $t_z, \tau_r,\tau_z$ can be estimated similarly. The only difference for $\tau_r,\tau_z$ is that the numerators in \eqref{eq: grad tau est 1} involve $\phi_t$, which can be estimated with an extra factor of $\tau$, as is seen in \eqref{eq: grad phi 1}.

\noindent{\bf Proof of \eqref{eq: grad phi 4}.}  We next obtain H\"older estimates of the first derivatives of $\phi$ and $\phi^{-1}$. %First consider \eqref{eq: grad phi 4}. 
It is easy to see that one just needs to prove it first for the case $\tau_1=\tau_2$, then for the case $t_1=t_2$. Assuming for the moment that $\tau_1=\tau_2=\tau$, we note that $\phi_\tau(t_2,\tau)-\phi_\tau(t_1,\tau)=\int_{t_1}^{t_2} \phi_{t\tau} (s,\tau)ds$  and use \eqref{eq: var ODE} for $\phi_{t\tau}$  to get
$$|\phi_\tau(t_2,\tau)-\phi_\tau(t_1,\tau)|\le \|V\|_{2}\|\phi_\tau\|_0|t_2-t_1|\le C|t_2-t_1|.$$
%$$|\phi(t_2,\tau)-\phi(t_2,\tau)|\le \|V\|_{C^2}\|\phi_\tau\|_\infty|t_2-t_1|\le C|t_2-t_1|.$$ 
We assume $t_1=t_2=t$ and use \eqref{eq: var ODE} again to get
\begin{align*}
\left|\frac{\partial}{\partial t}\phi_\tau(t,\tau_1)-\phi_\tau(t,\tau_2)\right|\le&~ \|V\|_{2}|\phi_\tau(t,\tau_1)-\phi_\tau(t,\tau_2)|\\
&\quad+\|V\|_{2,\beta}|\phi(t,\tau_1)-\phi(t,\tau_2)|^\beta\\
\le&~ \|V\|_{2}|\phi_\tau(t,\tau_1)-\phi_\tau(t,\tau_2)|\\
&\quad+C\|V\|_{2,\beta}|\tau_1-\tau_2|^\beta.
\end{align*}
Here the bound on $\phi(t,\tau_1)-\phi(t,\tau_2)$ can be obtained by using \eqref{eq: grad phi 1}  as $$|\phi(t,\tau_1)-\phi(t,\tau_2)| = \left| \int_{\tau_2}^{\tau_1} \phi_\tau(t, s)ds \right| \le C |\tau_1-\tau_2|.$$
  %by applying Gronwall and using \eqref{eq: ODE}. 
Now we integrate and use Gronwall to conclude
$$
|\phi_\tau(t,\tau_1)-\phi_\tau(t,\tau_2)|\le C|\tau_1-\tau_2|^\beta.
$$
This proves \eqref{eq: grad phi 4}.

\noindent{\bf Proof of \eqref{eq: grad phi 3}.} 
Let us now compute
\begin{align}
\frac1\tau r_t(t,\tau) &= \frac1\tau V_z(\phi(t,\tau))\notag\\
&=\int_0^1 [V_{zr}(\phi(t,s\tau))r_\tau(t,s\tau)+V_{zz}(\phi(t,s\tau))z_\tau(t,s\tau)]~ds,
\end{align}
for which we used $V_z(\phi(t,0))=0$. It follows that
\begin{align*}
&~\frac1{\tau_1} r_t(t_1,\tau_1) -\frac1{\tau_2} r_t(t_2,\tau_2) \\
=&~\int_0^1 [V_{zr}(\phi(t_1,s\tau_1))r_\tau(t_1,s\tau_1)-V_{zr}(\phi(t_2,s\tau_2))r_\tau(t_2,s\tau_2)]~ds\\
&\quad +\int_0^1[V_{zz}(\phi(t_1,s\tau_1))z_\tau(t_1,s\tau_1)-V_{zz}(\phi(t_2,s\tau_2))z_\tau(t_2,s\tau_2)]~ds,
\end{align*}
We can write similar expressions for $\frac1\tau z_t$. 
Inequality \eqref{eq: grad phi 3} now follows from the following estimates:
\begin{align*}
|\nabla^2 V(\phi(t_1,s\tau_1))-\nabla^2 V(\phi(t_2,s\tau_2))|&\le \|V\|_{2,\beta}|\phi(t_1,s\tau_1)-\phi(t_2,s\tau_2)|^\beta\\
&\le C\|\nabla \phi\|_0^\beta(|t_1-t_2|+|\tau_1-\tau_2|)^\beta\\
&\le C(|t_1-t_2|+|\tau_1-\tau_2|)^\beta,\\
|\phi_\tau(t,s\tau)|&\le C,\\
|\phi_\tau(t_1,s\tau_1)-\phi_\tau(t_2,s\tau_2)|&\le C(|t_1-t_2|+|\tau_1-\tau_2|^\beta).
\end{align*}
Note that the last estimate above is the previously proven  \eqref{eq: grad phi 4}.

\noindent{\bf Proof of \eqref{eq: grad phi 5}.}  We now recall \eqref{eq: t_r est 1} to get 
\eqn
(r^2+z^2)^{\frac{1+\beta}{2}}t_r(r,z)=\frac{(r^2+z^2)^{\frac{1+\beta}{2}}}{\tau}\frac{z_\tau}{\int_0^1V_{rr}(s\tau,0)~ds}.
\eeqn
Denoting by $(t_i,\tau_i)=\phi^{-1}(r_i,z_i)$, $i=1,2$, we have
\begin{align}
&~(r_1^2+z_1^2)^{\frac{1+\beta}{2}}t_r(r_1,z_1)-(r_2^2+z_2^2)^{\frac{1+\beta}{2}}t_r(r_2,z_2)\notag\\
=&~\frac{|\phi(t_1,\tau_1)|^{1+\beta}}{\tau_1}\frac{z_\tau(t_1,\tau_1)}{\int_0^1 V_{rr}(s\tau_1,0)~ds}-\frac{|\phi(t_2,\tau_2)|^{1+\beta}}{\tau_2}\frac{z_\tau(t_2,\tau_2)}{\int_0^1 V_{rr}(s\tau_2,0)~ds}\notag\\
=&~\frac{|\phi(t_1,\tau_1)|^{1+\beta}}{\tau_1}\left(\frac{z_\tau(t_1,\tau_1)}{\int_0^1 V_{rr}(s\tau_1,0)~ds}-\frac{z_\tau(t_2,\tau_2)}{\int_0^1 V_{rr}(s\tau_2,0)~ds}\right)\label{eq: grad phi Holder 1}\\
&\quad+\frac{z_\tau(t_2,\tau_2)}{\int_0^1 V_{rr}(s\tau_2,0)~ds}\left(\frac{|\phi(t_1,\tau_1)|^{1+\beta}}{\tau_1}-\frac{|\phi(t_2,\tau_2)|^{1+\beta}}{\tau_2}\right).\label{eq: grad phi Holder 2}
\end{align}
To estimate \eqref{eq: grad phi Holder 1}, we use \eqref{eq: phi close to tau} for the size of $\phi(t_1,\tau_1)$, \eqref{eq: grad phi 1}, \eqref{eq: grad phi 4} for $z_\tau(t_1,\tau_1)-z_\tau(t_2,\tau_2)$, and H\"older estimates on $V_{rr}$ for the integrand to infer that the size of \eqref{eq: grad phi Holder 1} is bounded by 
\eqn\label{eq: grad phi Holder 3}
C\tau_1^\beta(|t_1-t_2|+|\tau_1-\tau_2|^\beta).
\eeqn
Without loss of generality, assume $0<|(r_1,z_1)|\le|( r_2,z_2)|$. If $\frac12|( r_2,z_2)|\le |( r_1,z_1)|\le |( r_2,z_2)|$, use \eqref{eq: phi close to tau}, \eqref{eq: grad phi 2} to deduce  that \eqref{eq: grad phi Holder 3} is bounded by 
\begin{align}
&~C(r_1^2+z_1^2)^{\frac\beta2}\left(\frac{1}{\sqrt{r_1^2+z_1^2}}|(r_1,z_1)-(r_2,z_2)|+|(r_1,z_1)-(r_2,z_2)|^\beta\right)\notag\\
\le &~C|(r_1,z_1)-(r_2,z_2)|^\beta\left(\frac{|(r_1,z_1)-(r_2,z_2)|^{1-\beta}}{|(r_1,z_1)|^{1-\beta}}+|(r_1,z_1)|^\beta\right)\notag\\
\le &~C|(r_1,z_1)-(r_2,z_2)|^\beta.
\end{align}
If $|(r_1,z_1)|\le \frac12|(r_2,z_2)|$, then $|(r_1,z_1)|\le |(r_1,z_1)-(r_2,z_2)|$. In this case, \eqref{eq: grad phi Holder 3} is bounded by 
$$
C|(r_1,z_1)|^\beta\left(1+|(r_1,z_1)-(r_2,z_2)|^\beta\right)\le C|(r_1,z_1)-(r_2,z_2)|^\beta.
$$
Next we estimate the terms in the parentheses of \eqref{eq: grad phi Holder 2}. If $\frac12|( r_2,z_2)|\le |( r_1,z_1)|\le |( r_2,z_2)|$, we write them as 
\eqn
\frac{|(r_1,z_1)|^{1+\beta}-|(r_2,z_2)|^{1+\beta}}{\tau_1}+|(r_2,z_2)|^{1+\beta}\frac{\tau_2-\tau_1}{\tau_1\tau_2}, 
\eeqn
which  can be bounded by 
\begin{align}
&~C\left(\frac{|(r_1,z_1)|^\beta|(r_1,z_1)-(r_2,z_2)|}{|(r_1,z_1)|}+\frac{|(r_1,z_1)-(r_2,z_2)|}{|(r_2,z_2)|^{1-\beta}}\right)\notag\\
\le &~C|(r_1,z_1)-(r_2,z_2)|^\beta\frac{|(r_1,z_1)-(r_2,z_2)|^{1-\beta}}{|(r_1,z_1)|^{1-\beta}}\notag\\
\le&~C|(r_1,z_1)-(r_2,z_2)|^\beta.
\end{align}
If $|(r_1,z_1)|\le \frac12|(r_2,z_2)|$, the terms in the parentheses of \eqref{eq: grad phi Holder 2} is bounded by 
$$C\left(|(r_1,z_1)|^\beta+|(r_2,z_2)|^\beta\right)\le C|(r_2,z_2)|^\beta\le C|(r_1,z_1)-(r_2,z_2)|^\beta.$$
% by writing them as
%\eqn
%\frac{r_2}{\tau_2}(r_1-r_2)+\left(\frac{r_1}{\tau_1}-\frac{r_2}{\tau_2}\right)r_1+\frac{z_2}{\tau_2}(z_1-z_2)+\left(\frac{z_1}{\tau_1}-\frac{z_2}{\tau_2}\right)z_1
%\eeqn
%Now \eqref{eq: grad phi 1}, \eqref{eq: grad phi 3} show that the size of the above expression is bounded by 
%\eqn
%C[|r_1-r_2|+|z_1-z_2|+(r_1+z_1)(|t_1-t_2|^\beta+|\tau_1-\tau_2|^\beta)].
%\eeqn
%Use \eqref{eq: grad phi 2} again to bound the above by
%\begin{align*}
%&~C\left[|(r_1,z_1)-(r_2,z_2)|+(r_1+z_1)\left(\frac{|(r_1,z_1)-(r_2,z_2)|^\beta}{\sqrt{r_1^2+z_1^2}^\beta}+|(r_1,z_1)-(r_2,z_2)|^\beta\right)\right]\\
%\le &~C|(r_1,z_1)-(r_2,z_2)|^\beta.
%\end{align*}
The proof of \eqref{eq: grad phi 5} for $t_r$ is now complete. The estimates on $t_z$ are completely analogous. 

\noindent{\bf Proof of \eqref{eq: grad phi 6}.}  We write 
\eqn
(r^2+z^2)^{\frac\beta2}\tau_r(r,z) = (r^2+z^2)^{\frac\beta2}\frac{-\frac{z_t}{\tau}}{\int_0^1V_{rr}(s\tau,0)~ds}
\eeqn
and estimate as before. We only have to use \eqref{eq: grad phi 3} to replace \eqref{eq: grad phi 4} in the argument.
%\begin{align}
%\frac1\tau z_t(t,\tau) &= -\frac1\tau V_r(\phi(t,\tau))\notag\\
%&=-\int_0^1 [V_{rr}(\phi(t,s\tau))r_\tau(t,s\tau)+V_{zr}(\phi(t,s\tau))z_\tau(t,s\tau)~ds.
%\end{align}
\end{proof}

%%%%%%%%%%%%%%%
\subsection{Weighted H\"older spaces and $S$ estimates}

To facilitate the proof of Theorem \ref{lem: V to S map}, we introduce the following weighted H\"older spaces. For simplicity we denote $(r,z)$ by $x$. Given $\beta\in(0,1)$, $k\in \real$, for a continuous function $f$ on $B_R$, we define the norms
\eqn\label{eq: weighted Holder norm 1}
\|f\|_{C^{0,\beta}_{(k)}}=\sup_{x\in B_R}|x|^k|f(x)|+\sup_{x,y\in B_R,x\ne y}\frac{\left||x|^{k+\beta}f(x)-|y|^{k+\beta}f(y)\right|}{|x-y|^\beta}.
\eeqn
\eqn\label{eq: weighted Holder norm 2}
\|f\|_{C^{0,\beta}_{[k]}}=\sup_{x\in B_R}|x|^k|f(x)|+\sup_{x,y\in B_R,x\ne y}\min(|x|^{k+\beta},|y|^{k+\beta})\frac{\left|f(x)-f(y)\right|}{|x-y|^\beta}.
\eeqn
The spaces $C^{0,\beta}_{(k)}$ and $C^{0,\beta}_{[k]}$ are defined accordingly as sets of continuous functions with finite norms. 
We may now reformulate \eqref{eq: grad phi 2}, \eqref{eq: grad phi 5}, \eqref{eq: grad phi 6} as 
\eqn\label{eq: grad t tau space}
\nabla t\in C^{0,\beta}_{(1)}, ~\nabla \tau\in C^{0,\beta}_{(0)}.
\eeqn 
The spaces $C^{0,\beta}_{(k)}$ and $C^{0,\beta}_{[k]}$ actually coincide, as we show in the next lemma.  
\begin{lemma}\label{lem: equiv Holder}
For each $k$, there exists a constant $C>0$ such that 
\eqn
\frac1C\|f\|_{C^{0,\beta}_{(k)}}\le \|f\|_{C^{0,\beta}_{[k]}}\le C\|f\|_{C^{0,\beta}_{(k)}}.
\eeqn
\end{lemma}
\begin{proof}
First consider the case that $k+\beta\ge 0$. Assume $x,y$ are in the first octant and $|x|\le |y|$ without loss of generality, so that $\min(|x|^{k+\beta},|y|^{k+\beta})=|x|^{k+\beta}$. We estimate
\begin{align*}
&~\frac{|x|^{k+\beta}f(x)-|y|^{k+\beta}f(y)}{|x-y|^\beta}\\
=&~\frac{|x|^{k+\beta}(f(x)-f(y))}{|x-y|^\beta}+\frac{f(y)\left(|x|^{k+\beta}-|y|^{k+\beta}\right)}{|x-y|^\beta}\\
=&~(\text{I})+(\text{II})
\end{align*}
We have $$|\text{I}|\le \|f\|_{C^{0,\beta}_{[k]}}.$$ 
If $\frac{|y|}{2}\le |x|\le |y|$,
\begin{align*}
|\text{II}|\le C\frac{|f(y)||y|^{k+\beta-1}|x-y|}{|x-y|^\beta}\le C\sup_{y\in B_R}|y|^k|f(y)|.
\end{align*}
If $|x|\le \frac{|y|}{2}$,
$$|\text{II}|\le C|f(y)|\frac{|y|^{k+\beta}}{|y|^\beta}\le C\sup_{y\in B_R}|y|^k|f(y)|.$$
These estimates imply $\|f\|_{C^{0,\beta}_{(k)}}\le C\|f\|_{C^{0,\beta}_{[k]}}$.
The case when $k+\beta<0$ can be estimated in a similar way by just assuming $|x|\ge |y|$ in the above estimates instead. 

Let us now focus on the reverse direction and assume $k+\beta\ge 0$ again. Assuming $|x|\le |y|$ without loss of generality, we write
\begin{align*}
&~|x|^{k+\beta}\frac{f(x)-f(y)}{|x-y|^\beta}\\
=&~\frac{|x|^{k+\beta}f(x)-|y|^{k+\beta}f(y)}{|x-y|^\beta}+\frac{f(y)\left(|y|^{k+\beta}-|x|^{k+\beta}\right)}{|x-y|^\beta}\\
=&~(\text{I})+(\text{II}).
\end{align*}
Obviously
$$|\text{I}|\le \|f\|_{C^{0,\beta}_{(k)}},$$
and $(\text{II})$ can be estimated as above. The case $k+\beta<0$ is analogous.
\end{proof}

%What makes these weighted H\"older spaces useful is the following lemma.
The weighted H\"older spaces $C^{0,\beta}_{(k)}$ enjoy the following algebraic property, which will be useful for the  H\"older estimates of the entropy.  
\begin{lemma}\label{lem: w Holder prod}
For given $\beta\in (0,1)$, $k,l\in\real$, there exists a constant $C>0$ such that
\eqn
\|fg\|_{C^{0,\beta}_{(k+l)}}\le C\|f\|_{C^{0,\beta}_{(k)}}\|g\|_{C^{0,\beta}_{(l)}}.
\eeqn
\end{lemma}
\begin{proof}
One obviously has 
$$\left||x|^{k+l}f(x)g(x)\right|\le \sup_{x\in B_R}|x|^k|f(x)|\cdot\sup_{x\in B_R}|x|^l|g(x)|.$$
On the other hand, for $x,y\in B_R$, $x\ne y$,
\begin{align*}
&~\frac{|x|^{k+l+\beta}f(x)g(x)-|y|^{k+l+\beta}f(y)g(y)}{|x-y|^\beta}\notag\\
=&~\frac{|x|^kf(x)\left(|x|^{l+\beta}g(x)-|y|^{l+\beta}g(y)\right)}{|x-y|^\beta}+\frac{|y|^lg(y)\left(|x|^{k+\beta}f(x)-|y|^{k+\beta}f(y)\right)}{|x-y|^\beta}\\
&~+|x|^kf(x)|y|^lg(y)\frac{|y|^\beta-|x|^\beta}{|x-y|^\beta}.
\end{align*}
It is obvious that every term in the above sum is bounded by $C\|f\|_{C^{0,\beta}_{(k)}}\|g\|_{C^{0,\beta}_{(l)}}$.
\end{proof}

We now return to the proof of Theorem \ref{lem: V to S map}. Note the obvious bound $\|f\|_{0,\beta}\le \|f\|_{C^{0,\beta}_{(-\beta)}}$. Thus we only need to get $C^{0,\beta}_{(-\beta)}$ estimates on $\nabla S$.
\begin{proof}[\bf Proof of Theorem  \ref{lem: V to S map}]
It is easy to see that the unique solution $S$ to \eqref{eq: transport}, \eqref{eq: IC} is given by 
\eqn\label{eq: S sol}
S(\phi(t,\tau))=e^{\frac{\mu s_0(\tau)}{\gamma}}-\kappa\int_0^t (\omega^2 r)_z(\phi(t',\tau))~dt'.
\eeqn
More specifically, \eqref{eq: S sol} defines $S$ in the first quadrant, while the values of $S$ in the other quadrants are determined by symmetry. We observe that $S(r,z)\equiv 1$ near the boundary of $B_R$. In fact, by the set up of Theorem \ref{lem: V to S map} we know that $s_0(\tau)\equiv 0$ for $\tau>\frac{2R^0+4R}{6}$, $\omega^2(r,z)\equiv 0$ for $\sqrt{r^2+z^2}>\frac{2R^0+4R}{6}$, while \eqref{eq: phi close to tau} implies that 
\eqn\label{eq: tau phi comparable}
\sqrt{\frac{2R^0+4R}{R^0+5R}}\tau<|\phi(t,\tau)|<\sqrt{\frac{R^0+5R}{2R^0+4R}}\tau
\eeqn
if $\|V-V^0\|_{C^2}<\delta$ is sufficiently small. Now if $\sqrt{r^2+z^2}=|\phi(t,\tau)|>R_1=\tfrac{R^0+5R}{6}$, then $\tau>\sqrt{\tfrac{2R^0+4R}{R^0+5R}}R_1>\tfrac{2R^0+4R}{6}$, and $|\phi(t',\tau)|>\tfrac{2R^0+4R}{R^0+5R} R_1=\tfrac{2R^0+4R}{6}$. As a consequence, we see that $S(r,z)=1$ for $\sqrt{r^2+z^2}>R_1$.

By Proposition \ref{lem: char coord}, $\phi$ is a $C^1$ diffeomorphism. Thus $S(r,z)$ is obviously $C^1$ on $\Omega$. We will now obtain H\"older estimates for the first derivatives. We only show details for the $r$-derivative because the $z$-derivative is analogous.
\begin{align}
&~S_r(r,z)\notag \\
= &~\mu e^{\frac{\mu s_0(\tau)}\gamma}s_0'(\tau)\tau_r\label{eq: S_r est 1}\\
&\quad-\kappa(\omega^2 r)_z(r,z)t_r\label{eq: S_r est 2}\\
&\quad-\kappa \tau_r\int_0^t[(\omega^2r)_{zr}(\phi(t',\tau))r_\tau(t',\tau)+(\omega^2r)_{zz}(\phi(t',\tau))z_\tau(t',\tau)]~dt'.\label{eq: S_r est 3}
\end{align}
We estimate each term one by one. First of all, 
\begin{align*}
\left|e^{\frac{\mu s_0(\tau(r_1,z_1))}{\gamma}}-e^{\frac{\mu s_0(\tau(r_2,z_2))}{\gamma}}\right|&\le C\|s_0'\|_{0}\|\nabla \tau\|_0|(r_1,z_1)-(r_2,z_2)|\\
&\le C|(r_1,z_1)-(r_2,z_2)|
\end{align*}
by \eqref{eq: grad phi 2}. Thus $e^{\frac{\mu s_0(\tau)}\gamma}\in C^{0,1}(\overline{B_R})$. By \eqref{eq: grad t tau space} and Lemma \ref{lem: w Holder prod}, we only need to bound $s_0'(\tau(r,z))$ in $C^{0,\beta}_{(-\beta)}$ in order to estimate \eqref{eq: S_r est 1}. Indeed, since $s_0'(0)=0$,
$$
(r^2+z^2)^{-\frac\beta2}|s_0'(\tau(r,z))|\le \|s_0\|_{1,\beta}\frac{|\tau(r,z)|^\beta}{(r^2+z^2)^{\frac\beta2}}\le C
$$
by \eqref{eq: tau phi comparable}, and
\begin{align*}
|s_0'(\tau(r_1,z_1))-s_0'(\tau(r_2,z_2))|&\le C\|s_0\|_{1,\beta}\|\nabla \tau\|_0^\beta|(r_1,z_1)-(r_2,z_2)|^\beta\\
&\le C|(r_1,z_1)-(r_2,z_2)|^\beta.
\end{align*}
In order to estimate \eqref{eq: S_r est 2}, we must bound  $(\omega^2 r)_z$  %$(\omega^2 z)_r$ 
in $C^{0,\beta}_{(-1-\beta)}$, as can be seen from \eqref{eq: grad t tau space} and Lemma \ref{lem: w Holder prod}. By Lemma \ref{lem: equiv Holder}, we may compute the $C^{0,\beta}_{[-1-\beta]}$ norm of $(\omega^2 r)_z$  %$(\omega^2 z)_r$
 instead. In fact, we show that it belongs to $C^{0,1}_{[-2]}\subset C^{0,\beta}_{[-1-\beta]}$, as follows. 
$$
|(r,z)|^{-2}  |(\omega^2 )_z (r,z)r|   %|(\omega^2)_r(r,z)z|
\le \|\omega^2\|_{2}\frac{ r\sqrt{r^2+z^2}}{r^2+z^2}\le C.
$$
Let $r_1^2+z_1^2\ge r_2^2+z_2^2$.
\begin{align*}
&~|(r_1,z_1)|^{-1}\frac{\left|(\omega^2)_z(r_1,z_1)r_1-(\omega^2)_z(r_2,z_2)r_2\right|}{|(r_1,z_1)-(r_2,z_2)|}\\
\le &~ (r_1^2+z_1^2)^{-1/2}\|\omega^2\|_{2}\sqrt{r_1^2+z_1^2}\\
\le &~C.
\end{align*}
Now to estimate \eqref{eq: S_r est 3}, we must bound the integral term in $C^{0,\beta}_{(-\beta)}$. Indeed,
\eqn\label{eq: D2 omega}
|(\omega^2r)_{zr}(\phi(t',\tau))|\le \|\omega^2\|_{2}|\phi(t',\tau(r,z))|\le C\tau(r,z)\le C|(r,z)|
\eeqn 
by \eqref{eq: tau phi comparable}. The same estimate holds for $(\omega^2r)_{zz}$. Together with \eqref{eq: grad phi 1}, this implies that the integral in \eqref{eq: S_r est 3} is bounded by $C|(r,z)|\le C|(r,z)|^\beta$.  

We now turn to the H\"older estimate on the integral. Denote $(t_i,\tau_i)=\phi^{-1}(r_i,z_i)$, $i=1,2$. We assume $|(r_1,z_1)|\le |(r_2,z_2)|$ and write
\begin{align}
&~\int_0^{t_1}[(\omega^2r)_{zr}(\phi(t',\tau_1))r_\tau(t',\tau_1)+(\omega^2r)_{zz}(\phi(t',\tau_1))z_\tau(t',\tau_1)]~dt'\notag\\
&~-\int_0^{t_2}[(\omega^2r)_{zr}(\phi(t',\tau_2))r_\tau(t',\tau_2)+(\omega^2r)_{zz}(\phi(t',\tau_2))z_\tau(t',\tau_2)]~dt'\notag\\
=&~\int_{t_2}^{t_1}[(\omega^2r)_{zr}(\phi(t',\tau_1))r_\tau(t',\tau_1)+(\omega^2r)_{zz}(\phi(t',\tau_1))z_\tau(t',\tau_1)]~dt'\label{eq: S_r est 4}\\
&~+\int_0^{t_2}\big[(\omega^2r)_{zr}(\phi(t',\tau_1))r_\tau(t',\tau_1)-(\omega^2r)_{zr}(\phi(t',\tau_2))r_\tau(t',\tau_2)\label{eq: S_r est 5}\\
&~\qquad +(\omega^2r)_{zz}(\phi(t',\tau_1))z_\tau(t',\tau_1)-(\omega^2r)_{zz}(\phi(t',\tau_2))z_\tau(t',\tau_2)\big]~dt'.\label{eq: S_r est 6}
\end{align}
By \eqref{eq: D2 omega} and a similar estimate on $(\omega^2 r)_{zz}$, \eqref{eq: S_r est 4} is bounded by 
$$C|(r_1,z_1)||t_1-t_2|\le C|(r_1,r_2)|\frac{|(r_1,z_1)-(r_2,z_2)|}{|(r_1,z_1)|}\le C|(r_1,z_1)-(r_2,z_2)|,$$
where we have used \eqref{eq: grad phi 2} to estimate $\nabla t$. Now we write the integrand of \eqref{eq: S_r est 5} as
\begin{align}
&~[(\omega^2r)_{zr}(\phi(t',\tau_1))-(\omega^2r)_{zr}(\phi(t',\tau_2))]r_\tau(t',\tau_2)\label{eq: S_r est 7}\\
&~+(\omega^2r)_{zr}(\phi(t',\tau_1))[r_\tau(t',\tau_1)-r_\tau(t',\tau_2)].\label{eq: S_r est 8}
\end{align}
By \eqref{eq: grad phi 1} and \eqref{eq: grad phi 2}, we deduce that \eqref{eq: S_r est 7} is bounded by 
$$\|\omega^2\|_{2,\beta}\|\phi_\tau\|_0^{1+\beta}\|\nabla\tau\|_0^\beta|(r_1,z_1)-(r_2,z_2)|^\beta\le C|(r_1,z_1)-(r_2,z_2)|^\beta.$$
Also, by \eqref{eq: D2 omega}, \eqref{eq: grad phi 4}, \eqref{eq: grad phi 2}, we obtain that \eqref{eq: S_r est 8} is bounded by 
$$C|(r_1,z_1)||\tau_1-\tau_2|^\beta\le C\|\nabla\tau\|_0^\beta|(r_1,z_1)-(r_2,z_2)|^\beta\le C|(r_1,z_1)-(r_2,z_2)|^\beta.$$ 
Thus we conclude that \eqref{eq: S_r est 5} is bounded by $C|(r_1,z_1)-(r_2,z_2)|^\beta$. Similar estimates can be obtained for \eqref{eq: S_r est 6}. We have shown that the integral in \eqref{eq: S_r est 3} is in $C^{0,\beta}_{(-\beta)}$, and we have completed the $C^{1,\beta}$ estimates of $S$ on $B_R\cap\{r>0,z>0\}$. 

Finally, in order to extend the $C^{1,\beta}$ estimates of $S$ to $B_{R}$, we must show that $S_r(0,z)=0$, $S_z(r,0)=0$. In fact, by \eqref{eq: ODE}, \eqref{eq: Jacobian}, \eqref{eq: grad tau est 1}, 
$$\tau_r(r,z) = \frac{V_r(r,z)}{V_r(\tau(r,z),0)}.$$
So for $\tau>0$
$$\tau_r(0,z)=\frac{V_r(0,z)}{V_r(\tau(0,z),0)}=0$$
by the axisymmetry of $V$. On the other hand, we obviously have $(\omega^2r)_z(0,z)=0$. The formula of $S_r$ given by \eqref{eq: S_r est 1}, \eqref{eq: S_r est 2}, \eqref{eq: S_r est 3} now shows that $S_r(0,z)=0$. The proof that $S_z(r,0)=0$ is similar.
\end{proof}

%%%%%%%%%%%%%%%%%%%%%%%%%%
\section{Construction of the Iteration Map}\label{sec: iteration}
 Recall the definition of the mapping $\F$ in Section \ref{iteration subsection}.  
The next lemma establishes the basic properties of $\F$.
\begin{lemma} \label{lem: F well-defined}
Let $\beta\in (0,\min(q,1))$, $s_0\in C^{1,\beta}(\overline{B_R}\cap \{z=0\})$, and $\omega^2\in C^{2,\beta} (\overline {B_R})$. Then there exist $\kappa_0>0$, $\mu_0>0$ and a neighborhood $B_\delta(V^0)$ of $V^0$ in $C^{2,\beta}(\overline {B_R})$ such that $\F$ maps $B_\delta(V^0)\times \real\times (-\kappa_0,\kappa_0)\times (-\mu_0,\mu_0)$ into  a bounded subset of $C^{2,\beta}(\overline {B_R})\times \real$.  
%, and the image of $\F$ is bounded in $C^{2,\beta}(\overline {B_R})\times \real$.
%where $U$ is a neighborhood of $V^0$ in $C^{2,\beta}(\overline {B_R})$. 
\end{lemma}

\begin{proof}  We let %We take $U$ to be the neighborhood 
$B_\delta(V^0)$ be the neighborhood of $V^0$ in $C^{2,\beta}(\overline {B_R})$ specified  in Theorem \ref{lem: V to S map}. By that theorem, for each $V\in B_\delta(V^0)$ there exists a unique solution $S\in C^{1,\beta}(\overline{B_{R}})$  to  \eqref{eq: transport} and \eqref{eq: IC}. We fix such a pair $V$ and $S$. 

We have to  show that (i) the solution  $ V^\#$ to \eqref{eq: tilde V elliptic}, \eqref{eq: tilde V BC} exists, is unique and $V^\#\in  C^{2,\beta}(\overline {B_R})$ and  (ii) $\alpha^\#$ given in \eqref{eq: tilde alpha} is a well-defined finite number. The latter claim (ii) is clear since $V_+^qS^{-1}$ is integrable. To verify (i), we invoke the Schauder theory for second-order linear elliptic equations. By \eqref{S near 1}, %Lemma \ref{lem: S almost const}, 
if we limit $\kappa_0$ and $\mu_0$ suitably, 
then $S\in C^{1,\beta}(\overline{B_{R}})$ is bounded away from 0, the function $-4\pi V_+^{q}S^{-1}+\kappa \nabla\cdot(\omega^2 r e_r)$ belongs to $C^{0,\beta}(\overline{B_{R}})$, and  the function $\tfrac{1}{|\cdot|}*(V_+^qS^{-1})+\alpha$ belongs to $C^{2,\beta}(\overline {B_R})$.  By Corollary 6.9 of  \cite{GilbargTru}, the Dirichlet problem \eqref{eq: tilde V elliptic}, \eqref{eq: tilde V BC} has a unique solution $ V^\#\in  C^{2,\beta}(\overline {B_R})$ whose $C^{2,\beta}$ norm depends only on $\|S\|_{1,\beta}$, $\|V\|_{2,\beta}$ and $\|\omega^2\|_{2,\beta}$. 
\end{proof}

In the next lemma we estimate a typical semilinear term in a H\"older norm.  
\begin{lemma}\label{lem: Frechet}
Let $U,V\in C^1(\overline{B_R})$ and $\|V\|_1<1$. Let $g\in C^{1,\beta_1}([0,\|U\|_0+1])$ for some $0<\beta<\beta_1<1$.  Then there is a constant $C>0$ such that 
\begin{align}
\|g(U+V)-g(U)-g'(U)V\|_0&\le \|g\|_{1,\beta_1}\|V\|_0^{1+\beta_1}, \label{eq: Frechet C0}\\
\|g(U+V)-g(U)-g'(U)V\|_{0,\beta}&\le C\|g\|_{1,\beta_1}(\|U\|_1+1)^\beta \|V\|_{0,\beta}^{1+\beta_1-\beta}\label{eq: Frechet C beta}. 
\end{align}
\end{lemma}
\begin{proof}
Since $g\in C^{1}$, we can write
\begin{align}
[g(U+V)-g(U)-g'(U)V](x)&=\int_0^1[g'(U+tV)-g'(U)](x)~dt\cdot V(x)\notag\\
:&=I(x)\cdot V(x).
\end{align}
Since $\|U+tV\|_0\le \|U\|_0+1$ for $\|V\|_1<1$, we have
\begin{align}
\|I\|_0\le \|g\|_{1,\beta_1}\|V\|_0^{\beta_1},
\end{align}
so that \eqref{eq: Frechet C0} follows.
Also
\begin{align}
[g(U+V)-g(U)-g'(U)V]_{0,\beta}&\le [I]_{0,\beta}\|V\|_0+\|I\|_0[V]_{0,\beta}\notag\\
&\le [I]_{0,\beta}\|V\|_{0,\beta}+\|g\|_{1,\beta_1}\|V\|_0^{\beta_1}[V]_{0,\beta}.\label{eq: g Holder 1}
\end{align}
To estimate $[I]_{0,\beta}$, we compute
\begin{align}
|I(x)-I(y)|&=\bigg|\int_0^1[g'(U(x)+tV(x))-g'(U(y)+tV(y))]~dt\notag\\
&\qquad - [g'(U(x))-g'(U(y))]\bigg|\notag\\
&\le \|g\|_{1,\beta_1}\left(|U(x)+tV(x)-U(y)-tV(y)|^{\beta_1}+|U(x)-U(y)|^{\beta_1}\right)\notag\\
&\le C\|g\|_{1,\beta_1}(\|U\|_1+\|V\|_1)^{\beta_1}|x-y|^{\beta_1}. \label{eq: I Holder 1}
\end{align}
On the other hand, 
\eqn\label{eq: I Holder 2}
|I(x)-I(y)|\le 2\|I\|_0\le 2\|g\|_{1,\beta_1}\|V\|_0^{\beta_1}.
\eeqn
Since $\beta<\beta_1$, we can interpolate between \eqref{eq: I Holder 1} and \eqref{eq: I Holder 2} to get
\eqn
|I(x)-I(y)|\le C\|g\|_{1,\beta_1}(\|U\|_1+\|V\|_1)^\beta|x-y|^\beta\|V\|_0^{\beta_1-\beta}.
\eeqn
Thus
\eqn
[I]_{0,\beta}\le C\|g\|_{1,\beta_1}(\|U\|_1+1)^\beta\|V\|_0^{\beta_1-\beta},
\eeqn
and \eqref{eq: Frechet C beta} follows from \eqref{eq: g Holder 1}.
\end{proof}

Next we show that the mapping $\F$ is  Fr\'echet differentiable and also that $I-D\F$ is invertible in case  $\kappa=\mu=0$. 

\begin{lemma}\label{lem: Lambda invertible}
Let $\tfrac65<\gamma<2$, $\gamma\ne \tfrac43$. Fix $\beta\in\left(0, \min\left(q-1,1\right)\right)$. Then $\F(V,\alpha,0,0)$ is Fr\'echet differentiable with respect to  $(V,\alpha)$ on $C^{2,\beta}(\overline {B_R})\times \real$, and 
\eqn\label{eq: Lambda}
\Lambda = I-D_{(V,\alpha)}\F(V^0,\alpha^0,0,0)
\eeqn 
is invertible with a bounded inverse. Furthermore, $\Lambda^{-1}$ is also bounded in the $C^{1,\beta}(\overline {B_R})\times \real$ norm, that is, 
\eqn
\|\Lambda^{-1}(\delta V,\delta\alpha)\|_{C^{1,\beta}(\overline {B_R})\times \real}\le \|(\delta V,\delta\alpha)\|_{C^{1,\beta}(\overline {B_R})\times \real}
\eeqn
for all $(\delta V,\delta\alpha)\in C^{2,\beta}(\overline {B_R})\times \real$.
\end{lemma}
\begin{proof}
It is easy to see from \eqref{eq: transport} and \eqref{eq: IC} that $S\equiv 1$ in case $\kappa=\mu=0$. Thus by \eqref{eq: tilde V elliptic}, \eqref{eq: tilde V BC}, \eqref{eq: tilde alpha}, $( V^\#, \alpha^\#) = \mathcal F(V,\alpha,0,0)$ solves 
\begin{align}
\Delta V^\#&=-4\pi V_+^q \quad \text{ on }B_R,\label{eq: isentropic V tilde elliptic}\\
 V^\# &= \frac1{|\cdot|}*V_+^q+\alpha \quad \text{ on }\partial B_R,\label{eq: isentropic V tilde BC}\\
\alpha^\# &= \alpha+\int_{B_R} V_+^q~dx-M. \label{eq: isentropic alpha}
\end{align}
From \eqref{eq: isentropic V tilde elliptic}, \eqref{eq: isentropic V tilde BC}, it follows  that $  V^\# - \frac1{|\cdot|}*V_+^q-\alpha$ is harmonic on $B_R$ with zero boundary values. Thus 
\eqn
 V^\# = \frac1{|\cdot|}*V_+^q+\alpha
\eeqn
actually holds in the whole ball $B_R$. It is then not hard to show using Lemma \ref{lem: Frechet} that $(V,\alpha)\mapsto( V^\#, \alpha^\#)$ is Fr\'echet differentiable with 
\eqn\label{eq: F deriv 1}
 D_{(V,\alpha)} V^\#(\delta V,\delta \alpha)=\frac{q}{|\cdot|}*\left(V_+^{q-1}\delta V\right)+\delta\alpha,
\eeqn
\eqn\label{eq: F deriv 2}
D_{(V,\alpha)} \alpha^\#(\delta V,\delta \alpha)=\delta\alpha+q\int_{B_R}V_+^{q-1}\delta V~dx .
\eeqn
%A key step in showing this is to use fact that  
%$$\|f(U+V)-f(U)\|_{C^{0,\beta}}\le C\|V\|_{C^0}^{1-\beta/\beta'}$$
%with $C=C(\beta,\beta',\|f\|_{C^{0,\beta'}}, \|U\|_{C^1},\|V\|_{C^1})$, where $\beta'>\beta$.
It is also easy to see that the Fr\'echet derivative of $(V,\alpha)\mapsto(V^\#,\alpha^\#)$ is compact on $C^{2,\beta}(\overline{B_R})\times \real$, so that $\Lambda$ defined by \eqref{eq: Lambda} is Fredholm of index zero. As for invertibility, we only need to prove its injectivity. So assume that $(\delta V,\delta\alpha)$ satisfies $\Lambda(\delta V,\delta \alpha)=0$.  That is, 
\eqn\label{eq: delta V Lane Emden}
\delta V = \frac{q}{|\cdot|}*\left((V^0)_+^{q-1}\delta V\right)+\delta\alpha,
\eeqn
\eqn\label{eq: delta V Lane Emden 2} 
\int_{B_R}(V^0)_+^{q-1}\delta V~dx=0.
\eeqn
We must prove that these equations admit only the zero solution.  This will follow in a way similar to Lemma 4.3 of \cite{SW19}; the argument is also related to section 4 of \cite{SW17}. 

In fact, we extend $\delta V$ to the entire $\real^3$ by \eqref{eq: delta V Lane Emden}. It follows that
\eqn\label{eq: delta V 1}
\Delta (\delta V) = -4\pi q(V^0)_+^{q-1}\delta V.
\eeqn
Let $Y_{lm}(\theta)$ ($l=0,1,2,\dots$, $m=-l,\dots,l$, $\theta\in \mathbb{S}^2$) be the standard spherical harmonics. For any function $f(x)$ with $x\in \real^3$, denote its $(l,m)$ component by 
$$f_{lm}(|x|)=\langle f,Y_{lm}\rangle(|x|) = \int_{\mathbb S^2}f(|x|\theta)Y_{lm}(\theta)~d_{\mathbb S^2}\theta.$$
We will first show $(\delta V)_{lm}=0$ if $l\ge 1$, so that only a radial component remains. Indeed, taking the $(l,m)$ component of \eqref{eq: delta V 1}, and using 
$$
\Delta = \partial_{|x|}^2+\frac{2}{|x|}\partial_{|x|}+\frac1{|x|^2}\Delta_{\mathbb S^2},
$$
$$
\int_{\mathbb S^2}(\Delta_{\mathbb S^2}f)g~d_{\mathbb S^2}\theta=\int_{\mathbb S^2}f(\Delta_{\mathbb S^2}g)~d_{\mathbb S^2}\theta,
$$
$$
\Delta_{\mathbb S^2} Y_{lm} = -l(l+1)Y_{lm},
$$
we obtain 
\eqn\label{eq: delta V 2}
\Delta[ (\delta V)_{lm}]-\frac{l(l+1)}{|x|^2}(\delta V)_{lm}=-4\pi q(V^0)_+^{q-1}(\delta V)_{lm}.
\eeqn
Since $(V^0)_+=0$ for $|x|>R^0$, we can solve \eqref{eq: delta V 2} explicitly there and obtain
\eqn
(\delta V)_{lm}(x)=C|x|^{-(l+1)}+D|x|^l.
\eeqn
Since $\delta V$ is bounded near infinity, we obtain
\eqn\label{eq: delta V outside}
(\delta V)_{lm}(x)=C|x|^{-(l+1)}\quad \text{ if }|x|\ge R^0.
\eeqn
By Lemma \ref{lem: LE}, $(V^0)'=\partial_{|x|}V^0<0$ for $|x|>0$. We define for $|x|>0$
\eqn\label{def: psi}
\psi_{lm} = \frac{(\delta V)_{lm}}{(V^0)'}.
\eeqn
Note that for $l\ge 1$, 
\begin{align*}
|\psi_{lm}(x)| & := \left|\frac1{(V^0)'(x)}\int_{\mathbb S^2}\left[\delta V(|x|\theta)-\delta V(0)\right]Y_{lm}(\theta)~d_{\mathbb S^2}\theta\right|\\
&\le \frac{C|x|}{|(V^0)'(x)|}\sup_{|y|\le |x|}|\nabla (\delta V)(y)|.
\end{align*}
Inequality \eqref{eq: V0 near 0} and $\nabla (\delta V)(0)=0$ imply that
\eqn\label{eq: psi at zero}
\lim_{|x|\to 0^+}\psi_{lm}(x)=0.
\eeqn
By \eqref{eq: LE1}, 
\eqn\label{eq: V0 LE}
\Delta V^0=-4\pi (V^0)_+^q,
\eeqn
from which it follows that
\eqn\label{eq: V0 1}
\Delta [(V^0)']-\frac2{|x|^2}(V^0)'=-4\pi q(V^0)^{q-1}(V^0)'.
\eeqn
From \eqref{eq: delta V 2}, \eqref{def: psi}, \eqref{eq: V0 1} we get
\eqn\label{eq: psi 1}
\Delta \psi_{lm}+\frac{2\nabla(V^0)'\cdot \nabla \psi_{lm}}{(V^0)'}+\frac{2-l(l+1)}{|x|^2}(V^0)'\psi_{lm}=0.
\eeqn
If $l\ge 1$, then $2-l(l+1)\le 0$.  Letting $\Psi = \sup_{0<|x|<R^0}\psi_{lm}(x)>0$ and using the strong maximum principle on \eqref{eq: psi 1}, we know that $\Psi$ cannot be attained at any interior point on the punctured ball $0<|x|<R^0$. By the  condition \eqref{eq: psi at zero} at the origin, $\Psi$ can only be attained for $|x|=R^0$, so that $\psi_{lm}(R^0)=\Psi$. By the Hopf maximum principle, $\partial_{|x|}\psi_{lm}(R^0)>0$.  Now the $C^1$ continuity of $\psi_{lm}$ across the surface at $|x|=R^0$ and \eqref{eq: delta V outside}, \eqref{def: psi}, \eqref{eq: V0 LE} implies 
\eqn
0<\Psi=\psi_{lm}(R^0)=\frac{C}{(R^0)^{l+1}(V^0)'(R^0)},
\eeqn
\eqn
0<\partial_{|x|}\psi_{lm}(R^0)=\frac{C(1-l)}{(R^0)^{l+2}(V^0)'(R^0)}.
\eeqn
These equations force $1-l>0$, which contradicts the assumption that $l\ge 1$. We therefore conclude that $\Psi = \sup_{0<|x|<R^0}\psi_{lm}(x)\le0$. By similar reasoning, $\inf_{0<|x|<R^0}\psi_{lm}(x)\ge0$. It follows that $\psi_{lm}$, as well as $(\delta V)_{lm}$, vanish  on $B_{R^0}$. \eqref{eq: delta V outside} now implies $\delta V_{lm}=0$ everywhere.
%letting $u_0=V^0$, $\delta \rho = q(V^0)_+^{q-1}\delta V$, we obtain
%\eqn
%\delta \rho = q (V^0)_+^{q-1}\left(\frac{1}{|\cdot|}*\delta\rho+\delta\alpha\right),
%\eeqn
%\eqn
%\int \delta\rho(x)~dx=0.
%\eeqn
%These agree with (4.6) and (4.7) of [AWRapidly] if we take $h^{-1}(u)=u^q$. It follows from Lemma 4.3 of [AWRapidly] that $\delta\rho\equiv 0$, $\delta\alpha=0$. \eqref{eq: delta V Lane Emden} then implies $\delta V\equiv 0$.

Since $\delta V$ is now restricted to be a radial function, \eqref{eq: delta V 1} can be regarded as the following ODE (using $'$ to denote $\pa_{|x|}$):
\eqn
(\delta V)''+\frac2{|x|}(\delta V)' +4\pi q(V^0)_+^{q-1}\delta V =0.%,\quad 0<|x|<R^0.
\eeqn
We also have the obvious condition
\eqn
(\delta V)'(0)=0
\eeqn
due to symmetry. On the other hand, we note that by scaling symmetry of \eqref{eq: V0 LE}, $V(x;a) = a^{\frac2{q-1}}V^0(ax)$ solves
\eqn
\Delta V+4\pi V_+^q=0.
\eeqn
It follows that $U=\partial_a V(x;1)$ solves
$$
\Delta U+4\pi q(V^0)_+^{q-1}U=0,
$$
or
\eqn
U''+\frac2{|x|}U'+4\pi q(V^0)_+^{q-1}U=0,
\eeqn
with 
\eqn
U'(0)=0.
\eeqn
So $\delta V$ and $U$ satisfy the same ODE with vanishing derivative at zero. By uniqueness,  %theory for such initial value problems by converting the ODE to an integral equation, similar to the standard ODE theory. 
it follows that $\delta V = CU$ for some constant multiple $C$. 
Also,  $V(x;a)$ satisfies 
\eqn
\int_{\real^3} V_+^q~dx =a^{\frac{3-q}{q-1}}\int_{\real^3}(V^0)_+^q~dx.
\eeqn
Taking the derivative with respect to $a$ and setting $a=1$, we get
\eqn
\int_{\real^3}q(V^0)_+^{q-1}U~dx = \frac{3-q}{q-1}\int_{\real^3}(V^0)_+^q~dx\ne 0
\eeqn
if $q\ne 3$. In this case, \eqref{eq: delta V Lane Emden 2} and the condition $\delta V = CU$ obtained above imply that $C=0$.   Hence $\delta V$ and $\delta \alpha$ are zero. The proof is complete as we observe that $q=3$ if and only if $\gamma=\tfrac43$.

It remains to reprove the lemma in the weaker space $C^{1,\beta}(\overline{B_R})$. Observe that the right hand sides of \eqref{eq: F deriv 1} and \eqref{eq: F deriv 2} are well-defined for $\delta V\in C^{1,\beta}(\overline{B_R})$, so that we can extend the definition of $\Lambda$ to $C^{1,\beta}(\overline {B_R})\times \real$. Once again $\Lambda$ is Fredholm with index zero on $C^{1,\beta}(\overline {B_R})\times \real$. Thus we only need to show that \eqref{eq: delta V Lane Emden}, \eqref{eq: delta V Lane Emden 2} has unique zero solution assuming merely that $\delta V\in C^{1,\beta}(\overline{B_R})$. 
But due to the gain of regularity of the right hand side of \eqref{eq: delta V Lane Emden}, we recover $\delta V\in C^{2,\beta}(\overline{B_R})$. The result thus follows.
\end{proof}

We define the sequence of approximate solutions $(V_n, \alpha_n)$ by 
\begin{align}
(V_0,\alpha_0)&=(V^0,\alpha^0),\label{eq: initialization}\\
(V_{n+1},\alpha_{n+1})&=(V_n,\alpha_n)-\Lambda^{-1}[(V_n,\alpha_n)-\F(V_n,\alpha_n,\kappa,\mu)].\label{eq: iteration}
\end{align}
Here \eqref{eq: initialization} simply means that we use the Lane-Emden solution as the zeroth step in the iteration.

%%%%%%%%%%%%%%%%%%%%%%%%%%%%%%
\section{Convergence of Iterations and Uniqueness}\label{sec: convergence}
In this section, we fix $\beta<\beta_1:=\min(q-1,1)$ and will prove convergence of the iteration sequence in $C^{2,\beta'}(\overline{B_R})\times \real$ for every $0<\beta'<\beta$. The argument consists mainly of two steps. In the first step we show that all $(V_n,\alpha_n)$ remain in a small neighborhood of $(V^0,\alpha^0)$ in the $C^{2,\beta}(\overline{B_R})\times \real$ norm, so that the next term in the sequence is always well-defined by applying Lemma \ref{lem: F well-defined}. In the second step we show that the $C^{1,\beta}(\overline{B_R})\times \real$ norm of $(V_{n+1}-V_n,\alpha_{n+1}-\alpha_n)$ contracts, so that the iterates form a Cauchy sequence in the low norm. We then use  interpolation in the hierarchy of H\"older spaces to prove convergence.

To avoid clutter in the equations, we further simplify the notation as follows. Denote
$$g(V)=V_+^q$$
and 
$$
\mathcal F(V,\alpha,\kappa,\mu)\text{ by }\F(V,\alpha), \quad \mathcal F(V,\alpha,0,0)\text{ by }\F^0(V,\alpha).
$$
Recall that the Lane-Emden solution $(V^0,\alpha^0)$ is a fixed point: $\F^0(V^0,\alpha^0)=(V^0,\alpha^0)$. As before, we also use the notation $(V^\#,\alpha^\#)=\F(V,\alpha)$.
Moreover, we denote
$$D_{(V,\alpha)}\F(V^0,\alpha^0,0,0) \text{ by }D\F^0 $$
and 
$$(\delta V^\dagger,\delta \alpha^\dagger)=D\F^0(\delta V,\delta \alpha).$$
%So by \eqref{eq: Lambda} $$\Lambda=I-D\F^0.$$ 
From Lemma \ref{lem: Lambda invertible} and \eqref{eq: isentropic V tilde elliptic}, \eqref{eq: isentropic V tilde BC}, \eqref{eq: isentropic alpha} we may write
\begin{align}
\Delta (\delta V^\dagger)&=-4\pi g'(V^0)\delta V \quad \text{ on }B_R, \label{eq: dagger 1}\\
\delta V^\dagger &= \frac{1}{|\cdot|}*(g'(V^0)\delta V)+\delta\alpha \quad \text{ on }\partial B_R, \label{eq: dagger 2}\\
\delta \alpha^\dagger& = \delta\alpha+\int_{B_R}g'(V^0)\delta V~dx.
\end{align}
Thus we use $\#$ to denote components of the nonlinear mapping $\F$, and $\dagger$ to denote components of the linearized operator $D\F^0$. We also denote the exponentiated entropy $S$ constructed in Theorem \ref{lem: V to S map} by $\mathscr{S}(V)$ to emphasize its dependence on $V$. Finally, we denote by $X$ the space $C^{2,\beta}(\overline{B_R})\times \real$.

%%%%%%%%%%%%%%%%%%%%%
\subsection{Boundedness in the High Norm}
We now prove the uniform boundedness of the $X$ norm of the iterates.
\begin{lemma}\label{lem: bounded high norm}
For all $\epsilon>0$, $\epsilon_1>0$ sufficiently small, if $\|(V_n,\alpha_n)-(V^0,\alpha^0)\|_X<\epsilon$ and $|\kappa|+|\mu|<\epsilon_1$, then $\|(V_{n+1},\alpha_{n+1})-(V^0,\alpha^0)\|_X<\epsilon$.
\end{lemma}
\begin{proof}
We first take $\epsilon<\delta$ as given in Lemma \ref{lem: F well-defined}, so that $(V_{n+1},\alpha_{n+1})$ is well-defined. From \eqref{eq: iteration} and $\F^0(V^0,\alpha^0)=(V^0,\alpha^0)$ we get
\begin{align}  
&~(V_{n+1}-V^0,\al_{n+1}-\al^0) \notag\\
=&~ (V_n-V^0,\al_n-\al^0) 
- \Lambda^{-1}\{(V_n-V^0,\al_n-\al^0)    -  \F(V_n,\al_n)+\F^0(V^0,\al^0)\}\notag  \\
= &~(V_n-V^0,\al_n-\al^0)   - \Lambda^{-1}\{\Lambda(V_n-V^0,\al_n-\al^0) - \R_n\}  \notag\\
= &~\Lambda^{-1} \R_n,  \label{eq: iterate remainder}
\end{align}
where the remainder $\R_n:= (V_n^{\R},  \al_n^{\R})$  is defined as 
\begin{align}\label{remainder} 
(V_n^\R,\alpha_n^\R)&= \F(V_n,\al_n)-\F^0(V^0,\al^0) - D\F^0(V_n-V^0,\al_n-\al^0)\notag\\
&=(V^\#_n-V^0,\alpha^\#_n-\alpha^0)- ((V_n-V^0)^\dagger,(\al_n-\al^0)^\dagger).
% \notag\\
%&= \F(V_n,\al_n)-\F^0(V^0,\al^0) - D_V\F^0(V^0,\al^0) ((V_n-V^0) 
%-D_\al\F^0(V^0,\al^0)(\al_n-\al^0).  
\end{align}
By the definitions of the respective terms, we have
\begin{align}
\nabla \cdot(\mathscr S(V_n)\nabla V_n^\#)&=-\frac{4\pi g(V_n)}{\mathscr S(V_n)}+\kappa \nabla \cdot(\omega^2 re_r)\quad \text{ on }B_R, \label{eq: Vn sharp}\\
V_n^\#&=\frac{1}{|\cdot|}*\frac{g(V_n)}{\mathscr S(V_n)}+\alpha_n \quad \text{ on }\partial B_R\notag\\
\alpha_n^\# &= \alpha_n+\int_{B_R}\frac{g(V_n)}{\S(V_n)}~dx-M.\notag
\end{align}
\begin{align*}
\nabla \cdot(\nabla V^0)&=-4\pi g(V^0)\quad \text{ on }B_R,\\
V^0&=\frac{1}{|\cdot|}*g(V^0)+\alpha^0 \quad \text{ on }\partial B_R\\
\alpha^0 &= \alpha^0+\int_{B_R}g(V^0)~dx -M.
\end{align*}
\begin{align}
\nabla \cdot(\nabla (V_n-V^0)^\dagger)&=-4\pi g'(V^0)(V_n-V^0)\quad \text{ on }B_R,\label{eq: Vn dagger}\\
(V_n-V^0)^\dagger&=\frac{1}{|\cdot|}*(g'(V^0)(V_n-V^0))+\alpha_n-\alpha^0 \quad \text{ on }\partial B_R\notag\\
(\alpha_n-\alpha^0) ^\dagger&= \alpha_n-\alpha^0+\int_{B_R}g'(V^0)(V_n-V^0)~dx .\notag
\end{align}
Combining the preceding equations, we obtain equations for $(V_n^\R,\alpha_n^\R)$, namely, 
\begin{align}
\Delta V_n^\R=&~\nabla \cdot\left(\{1-\S(V_n)\}\nabla V_n^\#\right)-4\pi g(V_n)\left\{\frac{1}{\S(V_n)}-1\right\}\notag\\
&\quad-4\pi \left\{g(V_n)-g(V^0)-g'(V^0)(V_n-V^0)\right\}+\kappa\nabla \cdot(\omega^2 re_r)\notag\\
:=&~I_1+I_2+I_3+I_4\qquad \text{ in }B_R,\label{eq: Delta VR}
\end{align}
\begin{align}
V_n^\R &= \frac{1}{|\cdot|}*\left(\left\{\frac1{\S(V_n)}-1\right\}g(V_n)\right)+\frac1{|\cdot|}*\left\{g(V_n)-g(V^0)-g'(V^0)(V_n-V^0)\right\}\notag\\
&:= I_5+I_6\qquad \text{ on }\partial B_R,\label{eq: VR BC}
\end{align}
\begin{align}
\alpha_n^\R&= \int_{B_R}\left\{\frac{1}{\S(V_n)}-1\right\}g(V_n)~dx +\int_{B_R}\left\{g(V_n)-g(V^0)-g'(V^0)(V_n-V^0)\right\}~dx\notag\\
&:=I_7+I_8.
\end{align} 
Now let $\|(V_n,\alpha_n)-(V^0,\alpha^0)\|_X<\epsilon$ and $|\kappa|+|\mu|<\epsilon_1$, where $\epsilon$, $\epsilon_1$ are to be determined. We first estimate the $C^{2,\beta}(\overline{B_R})$ norm of $V_n^\R$ by means of Schauder estimates. This amounts to estimating the $C^{0,\beta}(\overline{B_R})$ norm of \eqref{eq: Delta VR} and the $C^{2,\beta}(\pa B_R)$ norm of \eqref{eq: VR BC}. Indeed, we have
\eqn
\|I_1\|_{0,\beta}\le \|V_n^\#\|_{2,\beta}\|\S(V_n)-1\|_{1,\beta}\le C\epsilon_1
\eeqn
by Lemma \ref{lem: F well-defined} and Theorem \ref{lem: V to S map}. Similarly,
\eqn
\|I_2\|_{0,\beta}\le C\|V_n\|_{0,\beta}\|\S(V_n)-1\|_{0,\beta}\le C\epsilon_1.
\eeqn
On the other hand, by Lemma \ref{lem: Frechet}, we have 
\eqn
\|I_3\|_{0,\beta}\le C\|V_n-V^0\|_{0,\beta}^{1+\delta_1}=C\epsilon^{1+\delta_1},
\eeqn
where $\delta_1=\beta_1-\beta$.
It is obvious that $\|I_4\|_{0,\beta}\le C\epsilon_1$. By standard potential estimates, the convolution with $\tfrac1{|x|}$ is bounded from $C^{0,\beta}(\overline{B_{R+1}})$ to $C^{2,\beta}(\overline{B_R})$. 
Because all the functions appearing under the convolution in \eqref{eq: VR BC} are supported in the interior of $B_R$, we only have to bound their $C^{0,\beta}(\overline{B_R})$ norms. It follows as before that $\|I_5\|_{2,\beta}\le C\epsilon_1$, and $\|I_6\|_{2,\beta}\le C\epsilon^{1+\delta_1}$. We now use the Schauder estimates to conclude that 
\eqn
\|V_n^\R\|_{2,\beta}\le C\epsilon_1+C\epsilon^{1+\delta_1}.
\eeqn
We also use the previous estimates directly to estimate $I_7$ and $I_8$, thereby obtaining 
\eqn
|\alpha_n^\R|\le C\epsilon_1+C\epsilon^{1+\delta_1}.
\eeqn
Estimation of \eqref{eq: iterate remainder} yields 
\begin{align*}
\|(V_{n+1},\alpha_{n+1})-(V^0,\alpha^0)\|_X&\le \|\Lambda^{-1}\|_{X\to X}C(\epsilon_1+\epsilon^{1+\delta_1})\\
&\le C(\epsilon_1+\epsilon^{1+\delta_1}).
\end{align*}
We now choose $\epsilon$ so small that $C\epsilon^{1+\delta_1}<\tfrac\epsilon2$, and $\epsilon_1$ so small that $C\epsilon_1<\tfrac\epsilon2$. The proof is complete.
\end{proof}

%%%%%%%%%%%%%%%%%%%%%
\subsection{Contraction in the Low Norm}
For notational convenience we denote $\S(V_n)$ by $S_n$. We denote by $Y$ the space $C^{1,\beta}(\overline{B_R})\times\real$. By Lemma \ref{lem: bounded high norm}, $V_n$ and $\alpha_n$ %$S_n$ 
are well-defined, and $\|V_n-V^0\|_{2,\beta}<\epsilon$ for all $n$ if $|\kappa|+|\mu|<\epsilon_1$. We want to prove that the $Y$-norm of $(V_n-V_{n-1},\alpha_n-\alpha_{n-1})$ decays. To that end, we first estimate the difference $S_n-S_{n-1}$.

\begin{lemma}\label{lem: diff_S} 
Let $S_n=\mathscr{S}(V_n)$ be given as above. Then
\eqn \label{est_diff_S0}
\|S_n - S_{n-1}\|_0\le C \|\nabla S_{n-1}\|_0 \|\nabla (V_n- V_{n-1})\|_0 
\eeqn
\eqn \label{est_diff_S}
\|S_n - S_{n-1}\|_{C^{0,\beta}_{[1-\beta]}} \le C \|\nabla S_{n-1}\|_\beta \|\nabla (V_n- V_{n-1})\|_{0,\beta} 
\eeqn
where $C^{0,\beta}_{[1-\beta]}$ denotes the weighted H\"{o}lder space defined in \eqref{eq: weighted Holder norm 2}.
\end{lemma}

\begin{proof} We note that the difference $S_n-S_{n-1}$ is the solution to the transport equation  
\eqn
(\pa_z V_n )\pa_r(S_n -S_{n-1}) -(\pa_r V_n) \pa_z (S_n-S_{n-1}) = h_n,  
\eeqn
where
\eqn 
h_n = (\pa_zS_{n-1} )\pa_r(V_n -V_{n-1}) - 
 (\pa_rS_{n-1} )\pa_z(V_n -V_{n-1}),
\eeqn
with the zero floor data $(S_n -S_{n-1})(r,0)=0$. Then $S_n-S_{n-1}$ satisfies 
\eqn\label{int_e}
(S_n - S_{n-1})(r,z)=\int_0^t h_n (\phi(t',\tau))~dt'.
\eeqn
where $(r,z)=\phi (t,\tau)$ is the $C^1$ diffeomorphism given by the characteristic coordinates associated with $V_n$, as in Section \ref{sec: char coord}. The $C^0$ estimate \eqref{est_diff_S0} follows directly from \eqref{int_e} since $t$ is bounded.  

To show \eqref{est_diff_S}, it suffices to estimate the weighted H\"{o}lder semi-norm. To this end, let $(r_1,z_1), (r_2,z_2)$ be given. We may assume $0<|(r_1,z_1)|\le |(r_2,z_2)|$ without loss of generality. %Then
%\eqn
%\begin{split}
%&|(r_1,z_1)|(S_n - S_{n-1})(r_1,z_1) - |(r_2,z_2)|(S_n - S_{n-1})(r_2,z_2) \\
%&= |(r_1,z_1)| \left\{ (S_n - S_{n-1})(r_1,z_1)  - (S_n - S_{n-1})(r_2,z_2)\right\} \\
%&\quad+ \{ |(r_1,z_1)| - |(r_2,z_2)|\} (S_n - S_{n-1})(r_2,z_2)\\
%&= (I) + (II)
%\end{split}
%\eeqn
%For $(II)$, by the triangle inequality, we have 
%\eqn\label{est_II}
%|(II)| \le | (r_1,z_1)- (r_2,z_2)| \|S_n-S_{n-1}\|_0 
%\eeqn
%To estimate $(I)$, 
By letting $(t_i,\tau_i)=\phi^{-1}(r_i,z_i)$, $i=1,2$, we see that 
\eqn
\begin{split}
&(S_n - S_{n-1})(r_1,z_1) - (S_n - S_{n-1})(r_2,z_2) \\
%&= \int_0^{t_1} h_n (\phi(t',\tau_1))~dt' - \int_0^{t_2} h_n (\phi(t',\tau_2))~dt' \\
 &= \int_{t_2}^{t_1} h_n (\phi(t',\tau_1))~dt' + \int_{0}^{t_2} h_n (\phi(t',\tau_1)) - h_n (\phi(t',\tau_2))~dt'\\
 & := I_1 + I_2
\end{split}
\eeqn
For the first integral $I_1$ we have by \eqref{eq: grad phi 2}
\begin{align}
|I_1| \le |t_1-t_2| \|h_n\|_0 \le C \frac{ | (r_1,z_1)- (r_2,z_2)|}{|(r_1,z_1)|}  \|h_n\|_0
\end{align}
For $I_2$, using the H\"{o}lder regularity of $h_n$ and \eqref{eq: grad phi 1}, \eqref{eq: grad phi 2}, we obtain 
\begin{align}
|I_2|&\le C \|\phi_\tau\|_0 ^\beta |\tau_1-\tau_2|^\beta\|h_n\|_{0,\beta}\notag\\
&\le C\|\nabla \tau\|_0^\beta | (r_1,z_1)- (r_2,z_2)|^\beta \|h_n\|_{0,\beta}\notag\\
&\le C | (r_1,z_1)- (r_2,z_2)|^{\beta} \|h_n\|_{0,\beta}.
\end{align}
Combining the above estimates, we deduce that 
\begin{align}
&~|(r_1,z_1)|\cdot \left|(S_n - S_{n-1})(r_1,z_1)  - (S_n - S_{n-1})(r_2,z_2)\right|\notag\\
\le &~C | (r_1,z_1)- (r_2,z_2)|^\beta  \|h_n\|_{0,\beta}. \label{est_I}
\end{align}
By the definition \eqref{eq: weighted Holder norm 2} of the $C^{0,\beta}_{[1-\beta]}$ norm,  we have proven \eqref{est_diff_S}. 
\end{proof}

We now write the recursive equation for the difference $(V_n-V_{n-1},\alpha_n-\alpha_{n-1})$. From \eqref{eq: iterate remainder}, we have 
\eqn\label{eq: recursion short}
(V_{n+1}-V_n,\al_{n+1}-\al_n) = \Lambda^{-1} (\R_n -\R_{n-1}), 
\eeqn 
where 
\begin{align}
\R_n -\R_{n-1} &= \F(V_n,\al_n)-\F(V_{n-1},\al_{n-1}) - D\F^0 (V_n-V_{n-1},\al_n-\al_{n-1})\notag\\
&=(V_n^\#-V_{n-1}^\#-(V_n-V_{n-1})^\dagger,\alpha_n^\#-\alpha_{n-1}^\#-(\alpha_n-\alpha_{n-1})^\dagger)\\
&:=(\W_n ,\alpha_n^\R-\alpha_{n-1}^\R). \label{def: W, alpha R}
\end{align}
Here we denoted the first component of $\R_n -\R_{n-1} $ by $\W_n$.
%\eqn
%\W_n = V_n^\# - V_{n-1}^\# - (V_n -V_{n-1})^\dagger.
%\eeqn
We obtain the equation for $\W_n$ from \eqref{eq: Vn sharp}, \eqref{eq: Vn dagger}:
\eqn\label{eq Wn}
\nabla \cdot (S_n \nabla \W_n ) = -\nabla \cdot H_n-  4\pi G_n \quad\text{ on }B_R ,
\eeqn
where 
\begin{align}\label{def H_n}
H_n =  \left\{S_n -1\right\} \nabla (V_n - V_{n-1})^\dagger +  \left[ S_n -S_{n-1}\right] \nabla V^\#_{n-1} 
\end{align}
and 
\eqn\label{def G_n}
G_n = \left\{  \frac{g(V_n)}{S_n} - \frac{g(V_{n-1})}{S_{n-1}}- g'(V^0) (V_n-V_{n-1}) \right\},
\eeqn
and on the surface $\pa B_R$ we have 
\begin{align}
\label{BC_Wn}\W_n = \frac1{|\cdot|} * G_n .
\end{align}
Similarly, we have 
\eqn   
\begin{split}\label{def: alpha n - alpha n -1}
\al_n^\R - \alpha_{n-1}^\R= \int_{B_R}  G_n  ~dx.
\end{split}
\eeqn  

%\begin{align}
%\nabla \cdot (S_n \nabla \W_n ) &= - \nabla \cdot \left(\left\{S_n -1\right\} \nabla (U_n - U_{n-1}) \right)- \nabla\cdot \left( \left[ S_n -S_{n-1}\right] \nabla V^\#_{n-1} \right)\label{source1}\\
%&\quad-\left\{  \frac{g(V_n)}{S_n} - \frac{g(V_{n-1})}{S_{n-1}}- g'(V^0) (V_n-V_{n-1}) \right\}\label{source2}
%\end{align}

\begin{lemma}\label{lem: contraction}
There exist $\ep>0, \ep_1>0$ such that if 
$\sup_{n\ge1} \|(V_n,\al_n) - (V^0,\al^0)\|_X < \ep$ for  $|\ka| + |\mu| < \ep_1$, then 
\eqn
\|(V_{n+1}- V_n,\alpha_{n+1}-\alpha_n)\|_{Y}\le \frac12  \|(V_n - V_{n-1},\alpha_n-\alpha_{n-1} )\|_Y 
\eeqn 
for all $n$.  
\end{lemma}

\begin{proof}
We first claim the estimates 
\eqn
\|H_n \|_{0,\beta} \le C \ep_1   \| (V_n - V_{n-1},\alpha_n-\alpha_{n-1}) \|_Y,
\eeqn
and
\eqn\label{eq: G_n est}
\| G_n \|_0 \le C( \ep_1 +\ep^{\beta_1})  \|  V_n - V_{n-1}\|_{1}  + C\| V_n - V_{n-1} \|_{0}^{1+\beta_1}.
\eeqn
%\texttt{...$G_n$ bounds are not optimal...}
Indeed, by \eqref{eq: dagger 1}, \eqref{eq: dagger 2} with $\delta V=V_n-V_{n-1}$, $\delta\alpha=\alpha_n-\alpha_{n-1}$ and standard potential theoretic estimates, we have
\begin{align}
\|(V_n-V_{n-1})^\dagger\|_{1,\beta} &\le C(\|V_n - V_{n-1}\|_0+|\alpha_n-\alpha_{n-1}|)\notag\\
&=C\|(V_n-V_{n-1},\alpha_n-\alpha_{n-1})\|_Y. 
\end{align}
We denote the two terms in \eqref{def H_n} by $H_{n1}$ and $H_{n2}$ respectively. We estimate
\begin{align}
\|H_{n1}\|_{0,\beta} &\le \|\left\{S_n -1\right\} \nabla (V_n - V_{n-1})^\dagger\|_{0,\beta}  \notag\\
&\le C\|S_n-1\|_{0,\beta}\|(V_n-V_{n-1})^\dagger\|_{1,\beta}\notag\\
&\le C\ep_1\|(V_n-V_{n-1},\alpha_n-\alpha_{n-1})\|_Y,
\end{align}
where we have used \eqref{S near 1}. %Lemma \ref{lem: S almost const} on $S_n-1$. 
On the other hand, 
\begin{align}
\|H_{n2}\|_{0,\beta} &\le \| \left[ S_n -S_{n-1}\right] \nabla V^\#_{n-1} \|_{0,\beta} \notag\\  % \left[|x|(S_n - S_{n-1})\right]_\beta \| \nabla V^\#_{n-1} \|_{1,0} 
&\le \| \left[ S_n -S_{n-1}\right] \nabla V^\#_{n-1} \|_{C^{0,\beta}_{[-\beta]}} \label{est: Hn2 1}\\
 &\le C \|S_{n}-S_{n-1}\|_{C^{0,\beta}_{[1-\beta]}}  \|\nabla V^\#_{n-1} \|_{C^{0,\beta}_{[-1]}} \label{est: Hn2 2}\\
&\le C\|\nabla (S_{n-1}-1)\|_{0,\beta}  \|V_n-V_{n-1}\|_{1,\beta}\label{est: Hn2 3}\\
&\le C\ep_1\|V_n-V_{n-1}\|_{1,\beta}.
\end{align}
Here we used the definition of the weighted H\"older spaces and Lemmas \ref{lem: equiv Holder} and \ref{lem: w Holder prod} regarding their properties to get \eqref{est: Hn2 1} and \eqref{est: Hn2 2}. We used Lemma \ref{lem: diff_S} and the estimate $ \|\nabla V^\#_{n-1} \|_{C^{0,\beta}_{[-1]}} \le\|\nabla V^\#_{n-1} \|_{1} \le C$ to get \eqref{est: Hn2 3}. That $V_{n-1}^\#$ is bounded in $C^2$ follows from the last assertion of Lemma \ref{lem: F well-defined}. Combining the above estimates, we infer that
\eqn
\|H_n\|_{0,\beta}\le C\epsilon_1\|(V_n-V_{n-1},\alpha_n-\alpha_{n-1})\|_Y.
\eeqn
We may rewrite $G_n$ as 
\begin{align*}
G_n=&~\left[g(V_n) - g(V_{n-1}) - g'(V_{n-1}) (V_n-V_{n-1})\right]   + \left\{ g'(V_{n-1}) - g'(V^0)\right\} (V_n-V_{n-1})  \\
&~+ \left\{ \frac{1}{S_n} -1\right\} [g(V_n) - g(V_{n-1})] + \left[ \frac{1}{S_n} - \frac{1}{S_{n-1}} \right] g(V_{n-1})\\
:= &~G_{n1}+G_{n2}+G_{n3}+G_{n4}.
\end{align*}
By \eqref{eq: Frechet C0},
\eqn
\|G_{n1}\|_0\le C\|V_{n}-V_{n-1}\|_0^{1+\beta_1}.
\eeqn
We also have
\eqn
\|G_{n2}\|_0\le C\|V_{n-1}-V^0\|_0^{\beta_1}\|V_n-V_{n-1}\|_0\le C\epsilon^{\beta_1}\|V_n-V_{n-1}\|_0.
\eeqn
From \eqref{S near 1}, %Lemma \ref{lem: S almost const},
\eqn
\|G_{n3}\|_0\le C\|S_n-1\|_0\|V_n-V_{n-1}\|_0\le C\epsilon_1\|V_n-V_{n-1}\|_0.
\eeqn
From Lemma \ref{lem: diff_S},
\begin{align}
\|G_{n4}\|_0\le C\|S_n-S_{n-1}\|_0&\le C\|S_{n-1}-1\|_1\|V_n-V_{n-1}\|_1\notag\\
&\le C\epsilon_1\|V_n-V_{n-1}\|_1.
\end{align}
\eqref{eq: G_n est} follows from combining the above estimates. 

By standard potential estimates, we have
\eqn
\left\|\frac1{|\cdot|}*G_n\right\|_{1,\beta}\le C\|G_n\|_0.
\eeqn
We now use the $C^{1,\beta}$ Schauder estimates given by Theorem 8.33 of \cite{GilbargTru} on the equations   \eqref{eq Wn}, \eqref{BC_Wn} to get
\eqn
\begin{split}
&~\|\W_n\|_{1,\beta} \\
\le&~ C (\|H_n \|_{0,\beta} + \| G_n \|_0 ) \\
%&\le C(\ep_1 +\ep^{\beta_1})  \|V_n - V_{n-1} \|_{1,\beta}+ C   \|V_n - V_{n-1} \|_{1,\beta}^{1+\delta}\\
\le &~C(\ep_1 +\ep^{\beta_1})  \|(V_n - V_{n-1},\alpha_n-\alpha_{n-1}) \|_Y+C\|V_n - V_{n-1} \|_{0}^{1+\beta_1}\\
\le &~C(\ep_1 +\ep^{\beta_1})  \|(V_n - V_{n-1},\alpha_n-\alpha_{n-1}) \|_Y
\end{split}
\eeqn
where  $\|V_n - V_{n-1}\|_0\le 2 \ep$, which follows from the assumption that $\|(V_n,\al_n) - (V^0,\al^0)\|_X < \ep$ for all $n\ge 1$. We also have from \eqref{def: alpha n - alpha n -1} that
\eqn
|\al_n^\R - \alpha_{n-1}^\R | \le C(\ep_1 +\ep^{\beta_1})  \|V_n - V_{n-1} \|_{1}
\eeqn
From \eqref{eq: recursion short}, \eqref {def: W, alpha R} and by Lemma \ref{lem: Lambda invertible} we get
\begin{align}
&~\|(V_{n+1}-V_n,\alpha_{n+1}-\alpha_n)\|_Y\notag\\
\le&~ C\|\Lambda^{-1}\|_{Y\to Y}(\ep_1 +\ep^{\beta_1})  \|(V_n - V_{n-1},\alpha_n-\alpha_{n-1}) \|_Y\notag\\
\le &~ C(\ep_1 +\ep^{\beta_1})  \|(V_n - V_{n-1},\alpha_n-\alpha_{n-1}) \|_Y.
\end{align}
We now choose $\epsilon_1$ so small that $C\epsilon_1<\tfrac14$, and $\epsilon$ so small that $C\ep^{\beta_1}<\tfrac14$, which completes the proof.
\end{proof}

%%%%%%%%%%%%%%%%%%%%%%%%%%
\subsection{Existence and Uniqueness}

We need the following interpolation lemma on H\"older spaces.
\begin{lemma}\label{lem: interpolation}
Let $0<\beta'<\beta$. There exists a $C>0$ such that
\eqn
\|u\|_{2,\beta'}\le C\|u\|_{1,\beta}^{\lambda_1}\|u\|_{2,\beta}^{1-\lambda_1},
\eeqn
\eqn
\|v\|_{1,\beta'}\le C\|v\|_{0}^{\lambda_2}\|v\|_{1,\beta}^{1-\lambda_2},
\eeqn
for all
$u\in C^{2,\beta}(\overline{B_R})$ and $v\in C^{1,\beta}(\overline{B_R})$. Here 
\begin{align*}
2+\beta' &= \lambda_1(1+\beta)+(1-\lambda_1)(2+\beta),\\
1+\beta' &= \qquad \qquad \quad ~(1-\lambda_2)(1+\beta).
\end{align*}
\end{lemma}
\begin{proof}
 See \cite{Alinhac} or  Ex. 3.2.6 in \cite{Krylov}.     \end{proof}

%\bcr Hongjie wrote: Regarding your question, when the domain is a convex cone, the multiplicative inequality is given as an exercise in Krylov's book "Lectures on elliptic and parabolic equations in Holder spaces". See Ex. 3.2.6 on page 37.  In the general case, I think it follows from a partition of unity argument. I don't have an explicit reference though.  \ec

We are now ready to prove the main existence theorem, which we repeat for the reader's convenience.  

\begin{theorem*}
Let $s_0\in C^{1,\beta}(\overline{B_R}\cap\{z=0\})$ and $\omega^2\in C^{2,\beta}(\overline{B_R})$ be given. There exist $\epsilon>0$, $\epsilon_1>0$ such that if $|\kappa|+|\mu|<\epsilon_1$, then there exists a unique solution $(V^*,\alpha^*)\in X$, and $S^*\in C^{1,\beta}(\overline{B_R})$ to \eqref{eq: curl}, \eqref{eq: div}, \eqref{eq: floor s}, \eqref{eq: BC}, \eqref{eq: mass}, with $\|(V^*,\alpha^*)-(V^0,\alpha^0)\|_X\le\epsilon$. Furthermore, the solution satisfies that $V^*_+$ is supported on $B_{\frac{R^0+R}{2}}$, and $S^*=1$ outside $B_{R_1}$ for some fixed $R_1\in (R^0,R)$.
\end{theorem*}
\begin{proof}
We take $\epsilon$ and $\epsilon_1$ small enough  that Lemmas \ref{lem: bounded high norm} and \ref{lem: contraction} hold. It follows by iteration that $$\|(V_n-V_{n-1},\alpha_n-\alpha_{n-1})\|_Y\le \frac C{2^n}.$$ Lemma \ref{lem: diff_S} now gives 
$$
\|S_{n}-S_{n-1}\|_0\le \frac{C}{2^n}.
$$
Pick a $\beta'\in(0,\beta)$. Lemmas \ref{lem: bounded high norm}, \ref{lem: interpolation} together with the preceding estimates imply that 
\eqn
\|V_{n}-V_{n-1}\|_{2,\beta'}\le \frac{C}{2^{\lambda_1 n}},
\eeqn
\eqn
\|S_n-S_{n-1}\|_{1,\beta'}\le \frac{C}{2^{\lambda_2 n}}.
\eeqn
It follows that $\{V_{n}\}$ is a Cauchy sequence in $C^{2,\beta'}(\overline{B_R})$, $\{S_n\}$ is a Cauchy sequence in $C^{1,\beta'}(\overline{B_R})$, and $\{\alpha_n\}$ is a Cauchy sequence in $\real$. Denote their limits by $V^*$, $S^*$ and $\alpha^*$. By \eqref{eq: iteration}, we have
\eqn
(V_n,\alpha_n)-\F(V_n,\alpha_n,\kappa,\mu)=\Lambda (V_{n+1}-V_n,\alpha_{n+1}-\alpha_n).
\eeqn
Since $\Lambda$ is bounded on $C^{2,\beta'}(\overline{B_R})\times \real$ by Lemma \ref{lem: Lambda invertible}, $(V_n,\alpha_n)-\F(V_n,\alpha_n,\kappa,\mu)$ converges to zero, and $(V_n^\#,\alpha_n^\#)=\F(V_n,\alpha_n,\kappa,\mu)$ converges to $(V^*,\alpha^*)$ in $C^{2,\beta'}(\overline{B_R})\times \real$. We know that $V_n^\#,\alpha_n^\#$ and $S_n$ satisfy the equations 
\begin{align*}
(S_n)_r(V_n)_z-(S_n)_z(V_n)_r&=-\kappa(\omega^2 r)_z \quad \text{ on }B_R,\notag\\
S_n(r,0)&=e^{\frac{\mu s_0}{\gamma}},\\
\nabla \cdot(S_n\nabla V_n^\#)&=-\frac{4\pi g(V_n)}{S_n}+\kappa \nabla \cdot(\omega^2 re_r)\quad \text{ on }B_R, \label{eq: Vn sharp}\\
V_n^\#&=\frac{1}{|\cdot|}*\frac{g(V_n)}{S_n}+\alpha_n \quad \text{ on }\partial B_R\notag\\
\alpha_n^\# &= \alpha_n+\int_{B_R}\frac{g(V_n)}{S_n}~dx-M.\notag
\end{align*}
Taking the limit as $n$ tends to infinity, we see that \eqref{eq: curl}, \eqref{eq: div}, \eqref{eq: floor s}, \eqref{eq: BC}, \eqref{eq: mass} are satisfied by $V^*,\alpha^*,S^*$. 

To see that $V^*\in C^{2,\beta}(\overline{B_R})$ and $\|(V^*,\alpha^*)-(V^0,\alpha^0)\|_X\le \ep$, we recall from Lemma \ref{lem: bounded high norm} that $\|(V_n,\alpha_n)-(V^0,\alpha_0)\|_X< \epsilon$. So
$$
\|V_n-V^0\|_2+\sum_{i,j}\frac{|\partial_i\pa_j [V_n-V^0](x)-\pa_i\pa_j [V_n-V^0](y)|}{|x-y|^\beta}+|\alpha_n-\alpha^0| <\epsilon.
$$
Taking the limit as $n$ tends to infinity, we get $V^*\in C^{2,\beta}(\overline{B_R})$ and 
$\|(V^*,\alpha^*)-(V^0,\alpha^0)\|_X\le \ep$.  Choosing $\ep< \sup_{|x|>(R^0+R)/2} V^0(x)$, 
and recalling that $V^0$ is positive precisely on $B_{R^0}$, 
it follows that $V_+^*$ is supported on $B_{{(R^0+R)}/{2}}$ (see Remark \ref{lem: support}).  
That $S^*=1$ outside $B_{R_1}$ was already proven in Theorem  \ref{lem: V to S map}. 

Finally in order to prove uniqueness, suppose there are two such 
fixed points $(V_1^*,\alpha_1^*,S_1^*)$ and $(V_2^*,\alpha_2^*,S_2^*)$. We can write $S^*_i=\S(V^*_i)$, and 
\eqn\label{eq: unique}
(V^*_i,\alpha^*_i)=(V^*_i,\alpha^*_i)-\Lambda^{-1}[(V^*_i,\alpha^*_i)-\F(V^*_i,\alpha^*_i)]
\eeqn
for $i=1,2$. We can now simply repeat on \eqref{eq: unique} the low-norm estimates on \eqref{eq: iteration} obtained in Lemma \ref{lem: contraction} to get 
\eqn
\|(V^*_1-V^*_2,\alpha^*_1-\alpha^*_2)\|_Y\le \frac12\|(V^*_1-V^*_2,\alpha^*_1-\alpha^*_2)\|_Y.
\eeqn
Thus $\|(V^*_1-V^*_2,\alpha^*_1-\alpha^*_2)\|_Y=0$ and the two solutions coincide. 
\end{proof}

%\section{First Order Equations}
\vspace{.2 in}

\noindent{\bf Acknowledgements.}  JJ is supported in part by the NSF DMS-grant 2009458. YW is supported in part by the NSF DMS-grant 2006212.

\end{document}